\documentclass[10pt,reqno]{article}

\usepackage{amssymb}
\usepackage{epsfig}
\usepackage{amsmath}
\usepackage{amsthm}
\usepackage{color}
\definecolor{r}{rgb}{0.9,0.3,0.1}
\definecolor{b}{rgb}{0.1,0.3,0.9}

\newtheorem{theorem}{Theorem}[section]
\newtheorem{lemma}[theorem]{Lemma}
\newtheorem{corollary}[theorem]{Corollary}

\theoremstyle{remark}
\newtheorem{remark}[theorem]{Remark}

\theoremstyle{definition}
\newtheorem{assumption}[theorem]{Assumption}

\newtheorem{definition}[theorem]{Definition}

\newcommand\cbrk{\text{$]$\kern-.15em$]$}}
\newcommand\opar{\text{\,\raise.2ex\hbox{${\scriptstyle
|}$}\kern-.34em$($}}
\newcommand\cpar{\text{$)$\kern-.34em\raise.2ex\hbox{${\scriptstyle |}$}}\,}

\newcommand{\al}{\alpha}

\newcommand{\ga}{\gamma}

\newcommand\bL{\mathbb{L}}
\newcommand\bR{\mathbb{R}}
\newcommand\bH{\mathbb{H}}
\newcommand\bZ{\mathbb{Z}}

\newcommand\bM{\mathbb{M}}

\newcommand\cB{\mathcal{B}}

\newcommand\cF{\mathcal{F}}

\newcommand\cL{\mathcal{L}}

\newcommand\cM{\mathcal{M}}

\newcommand\frH{\mathfrak{H}}

\newcommand\aint{-\hspace{-0.38cm}\int}

\newcommand{\mysection}[1]{\section{#1}
\setcounter{equation}{0}}

\begin{document}

\setlength{\baselineskip}{16pt}

\title
{A weighted $L_p$-theory for  second-order elliptic and parabolic partial differential systems on a half space}

\author{Kyeong-Hun Kim$\,^1$, Kijung Lee$\,^2$
\\[0.5cm]
\small{$\;^{1}\,$Department of Mathematics, Korea University}\\
\small{Seoul, 136-701, Korea}\\
\small{$\;^{2}\,$Department of Mathematics, Ajou University}\\
\small{Suwon, 443-749, Korea}\\}
\date{}


\maketitle

{\hspace{3.6cm}{{\it - Dedicated to 70th birthday of N.V. Krylov }}}
\vspace{0.4cm}

\begin{abstract}
In this paper we develop a Fefferman-Stein theorem, a Hardy-Littlewood theorem and   sharp function estimations in weighted Sobolev spaces.  We also  provide
 uniqueness and existence results for  second-order elliptic and parabolic partial differential systems in weighed Sobolev spaces.

\vspace*{.125in}


\vspace*{.125in}

\end{abstract}



\mysection{Introduction}

In this article  we consider  the elliptic system
\begin{eqnarray}
                              \label{main system elliptic}
\sum_{i,j=1}^{d}\sum_{r=1}^{d_1}a^{ij}_{kr}u^r_{x^ix^j}(x)=f^k(x), \quad (k=1,2,\cdots,d_1)
 \end{eqnarray}
and the parabolic system
\begin{eqnarray}
 u^k_t(t,x)=\sum_{i,j=1}^{d}\sum_{r=1}^{d_1}a^{ij}_{kr}(t)u^r_{x^ix^j}(t,x)+f^k(t,x), \quad  (k=1,2,\cdots,d_1) \label{main system}
 \end{eqnarray}
defined for $t>0$ and $x\in \bR^d_+$.

In the study of  partial differential equations (PDEs) or  of partial differential systems (PDSs)  regularity theory play the key role of describing
 essential relations between input data and the unknown solutions; the sharper the theory is, the more understanding of the relations we get.

The primary goals of this article are to introduce some new mathematical
tools and ideas which are useful  in the study of   systems in  $L_p$-spaces involving weights and to
 provide another
nice regularity theory for these systems.

In this article we use weighted Sobolev spaces  for the unknown function $u=(u^1,\cdots,u^{d_1})$ and the inputs $f^k$.
The need to introduce weights comes from, for instance, the theory of stochastic partial
differential equations (SPDEs) or stochastic partial differential systems (SPDSs), where a H\"older space approach does not allow us to obtain results of reasonable generality and Sobolev spaces without weights are trivially inappropriate (see \cite{KL2} for details).  To study  such stochastic systems one has to develop a nice regularity theory for the corresponding deterministic systems in advance.
Also Sobolev spaces with weights are very useful  in treating degenerate elliptic and parabolic equations (see, for instance, \cite{L}) and
in studying equations defined on non-smooth domains such as domains with wedges (see, for instance, \cite{G,L, S}).

 In principle there are three main methods for $L_p$-theory: multiplier theory, Calder\'on-Zygmund theory
 and the pointwise estimate using  sharp functions.
Multiplier theory fits well  when the principal operator is almost
Laplacian and the equation under consideration is defined on the entire space, and Calder\'on-Zygmund theory works well when there exists an integral representation of solutions and the integral is taken over
$\bR^n$ for some $n$. However, these two methods do not
fit  our case since we are dealing with weighted $L_p$-theories for systems (\ref{main system elliptic}) and (\ref{main system elliptic}) defined on a half space.  Thus  we use an approach based on
  pointwise estimates of the sharp function of second order derivatives, but unlike the standard
theory (for instance, \cite{kr08}) we need to use the weighted version.
The elaboration  of this approach   is one of our main results.

We also mention that if $d_1=1$ then weighted $L_p$-theories for single equations defined on a half space can be constructed based on integration by parts without relying on sharp function estimations (see the proof of Lemma 4.8 and Lemma 6.3 of \cite{kr99}). However it seems that the arguments in the proof of Lemma 4.8 and Lemma 6.3 of \cite{kr99} cannot be reproduced for $L_p$-theory of systems unless $p=2$ and some stronger algebraic conditions on $A^{ij}$
 are additionally assumed.

Interestingly, we discovered  some very useful tools in the perspective
of linear Partial differential equations/systems theory.
Even though, in this article, we only consider the systems with
coefficients independent of $x$, the sharp function estimates and
the tools used to derive them will naturally lead to many
subsequent works studying, for instance,  elliptic and
parabolic equations and systems with discontinuous coefficients defined in an arbitrary domain $U$ of $\bR^d$. In this context, we refer the readers to very extensive literature \cite{kr08} and recent articles \cite{B, By, C, Doy, H} (also see the references therein),
where (standard) $L_p$-theories are constructed for single equations with VMO (or small BMO)-coefficients.

The article is organized as follows. In section \ref{FS HL} we
prove the  Fefferman-Stein theorem and Hardy Littiewood theorem with
our special weights; the proofs are quite elementary. In section
\ref{main result} we introduce weighted Sobolve spaces and formulate our regularity results for the systems,   Theorem \ref{main theorem} and Theorem \ref{main theorem-elliptic}.  The \emph{useful}
tools and ideas for proving Theorem \ref{main theorem} and Theorem \ref{main theorem-elliptic} are in
section \ref{local estimate} and \ref{section sharp}; the local
estimations and the sharp function estimations.  Finally Theorem \ref{main theorem} and Theorem \ref{main theorem-elliptic} are proved in
section \ref{section proof}.

As usual $\bR^{d}$
stands for the Euclidean space of points $x=(x^{1},...,x^{d})$ and
$\bR^{d}_{+}=\{x\in\bR^{d}:x^{1}>0\}$.
For $i=1,...,d$, multi-indices $\alpha=(\alpha_{1},...,\alpha_{d})$,
$\alpha_{i}\in\{0,1,2,...\}$, and functions $u(x)$ we set
$$
u_{x^{i}}=\frac{\partial u}{\partial x^{i}}=D_{i}u,\quad
D^{\alpha}u=D_{1}^{\alpha_{1}}\cdot...\cdot D^{\alpha_{d}}_{d}u,
\quad|\alpha|=\alpha_{1}+...+\alpha_{d}.
$$
By $\delta^{kr}$ we
denote the Kronecker delta on the indices $k,r$. If we write  $N=N(\cdots)$, this means that the constant $N$ depends
only on what are in parenthesis.

The authors are sincerely grateful to Ildoo Kim for finding few errors in the earlier version of this article.

\mysection{F-S and H-L theorems  in  weighted $L_p$-spaces}
                            \label{FS HL}

Denote
$$
\Omega=\bR \times \bR^d_+:=\{(t,x)=(t,x^1,x^2,\ldots,x^d)\;:x^1>0\}.
$$
 Also, by $\cB(\bR^d_+)$ and $\mathcal{B}(\Omega)$ we denote the Borel
$\sigma$-algebra on $\bR^d_+$ and $\Omega$ respectively.  Fix $\alpha \in (-1,\infty)$ and define the weighted measures
$$
\nu(dx)=\nu_{\alpha}(dx)=(x^1)^{\alpha}dx, \quad d\mu=\mu_{\alpha}(dtdx):=\nu_{\alpha}(dx)dt.
$$
Then $(\bR^d_+, \cB(\bR^d_+), \nu)$ and $(\Omega,\cB(\Omega),\mu)$ are
measure spaces with $\nu(\bR^d_+)=\mu(\Omega)=\infty$. Let $p\in [1,\infty)$ and
$L_p(\Omega,\mu)=L_p(\Omega,\mu;\mathbb{R}^{d_1})$ ($L_p(\bR^d_+,\nu)$ resp.) be the collection
of Borel-measurable functions $u=(u^1,\ldots,u^{d_1})$ defined on $\Omega$ (on $\bR^d_+$ resp.) satisfying
\begin{eqnarray}
\|u\|^p_{L_p(\Omega,\mu)}:=\int_{\Omega}|u|^p d\mu <\infty, \quad \quad \left(\|u\|^p_{L_p(\bR^d_+,\nu)}:=\int_{\bR^d_+}|u|^p \nu(dx)<\infty, \text{respectively}\right).\nonumber
\end{eqnarray}
Denote
$$
\cB^0(\Omega):=\{C\in \cB(\Omega)\;:\; |C|:=\mu(C)
<\infty\}, \quad \cB^0(\bR^d_+):=\{D\in \cB(\bR^d_+)\;:\; |D|:=\nu(D)
<\infty\}.
$$
We say $f \in L_{1,loc}(\Omega,\mu;\mathbb{R}^{d_1})$ if
$fI_C \in L_1(\Omega,\mu)$ for any $C\in \cB^0(\Omega)$, where $I_C$ is the indicator function of $C$. For $f=(f^1,\ldots,f^{d_1})\in
L_1(\Omega,\mu;\mathbb{R}^{d_1})$ and $C \in\cB^0(\Omega)$ we
define
\begin{eqnarray}
f_C:=\frac{1}{|C|}\int_C f d\mu=\aint_C f d\mu=\left(\aint_C f^1
d\mu,\ldots,\aint_C f^{d_1}d\mu\right).\nonumber
\end{eqnarray}
 Similarly write $h\in L_{1,loc}(\bR^d_+,\nu;\mathbb{R}^{d_1})$ if $hI_D\in L_1(\bR^d_+,\nu)$ for any $D\in \cB^0(\bR^d_+)$, and define
$$
h_D:=\frac{1}{|D|}\int_D h \nu(dx)=\aint_D h \nu(dx)=\left(\aint_D h^1
\nu(dx),\ldots,\aint_D h^{d_1}\nu(dx)\right).
$$
Let $(\mathbb{C}_n, n\in \mathbb{Z})$  denote the filtration of the partitions of $\bar\Omega$ defined
  by
\begin{eqnarray}
\mathbb{C}_n=\Big{\{}\Big[\frac{i_0}{4^n},\frac{i_0+1}{4^n}\Big)\times
\Big[\frac{i_1}{2^n},\frac{i_1+1}{2^n}\Big)\times
 \cdots\times
\Big[\frac{i_d}{2^n},\frac{i_d+1}{2^n}\Big) \;:\;
i_0,i_2,\ldots,i_d\in \mathbb{Z},\;i_1\in \{0\}\cup
\mathbb{N}\Big{\}},\nonumber
\end{eqnarray}
and $(\mathbb{D}_n, n\in \mathbb{Z})$ be the corresponding filtration of the partitions of $\bar{\bR}^d_+$, that is,
\begin{eqnarray}
\mathbb{D}_n:=\Big{\{}
\Big[\frac{i_1}{2^n},\frac{i_1+1}{2^n}\Big)\times
 \cdots\times
\Big[\frac{i_d}{2^n},\frac{i_d+1}{2^n}\Big) \;:\;
i_0,i_2,\ldots,i_d\in \mathbb{Z},\;i_1\in \{0\}\cup
\mathbb{N}\Big{\}}.\nonumber
\end{eqnarray}
For any $(t,x)\in\Omega$, by $C_n(t,x)$ ($D_n(x)$ resp.) we  denote the unique cube in $\mathbb{C}_n$ (in $\mathbb{D}_n$ resp.) containing $(t,x)$ ($x$ respectively).
Let $\mathbb{L}=\bL(\Omega)$ (resp. $\bL(\bR^d_+)$) denote the
set of $\mathbb{R}^{d_1}$-valued {\bf continuous} functions with
compact support in $\Omega$ ( in $\bR^d_+$ respectively).

\begin{lemma}\label{filtration1}
(i) We have\; $\inf_{C\in \mathbb{C}_n}|C|\to \infty$ as $n\to -\infty$
and, for any $f\in \mathbb{L}(\Omega)$,
$\lim_{n\to\infty}f_{C_n(t,x)}=f(t,x)$ holds for any $(t,x)\in
\Omega$.

(ii) We have\; $\inf_{D\in \mathbb{D}_n}|D|\to \infty$ as $n\to -\infty$
and, for any $f\in \mathbb{L}(\bR^d_+)$,
$\lim_{n\to\infty}f_{D_n(x)}=f(x)$ holds for any $x\in
\bR^d_+$.
\end{lemma}
\begin{proof}
It is obvious since $f$ is continuous.
\end{proof}
\begin{lemma}\label{filtration2}
(i) For any $C\in \mathbb{C}_n$ there exists a unique $C'\in
\mathbb{C}_{n-1}$ such that $C\subset C'$ and
\begin{eqnarray}
\frac{|C'|}{|C|}\le N(\alpha)<\infty.\nonumber
\end{eqnarray}

(ii) For any $D\in \mathbb{D}_n$ there exists a unique $D'\in
\mathbb{D}_{n-1}$ such that $D\subset D'$ and
\begin{eqnarray}
\frac{|D'|}{|D|}\le N(\alpha)<\infty.\nonumber
\end{eqnarray}
\end{lemma}
\begin{proof}
We only prove (i).  Since $\mathbb{C}_{n-1}$ is a partition of $\Omega$, only one member
of it contains $C$; we call it $C'$. Let
\begin{eqnarray}
C'=\Big[\frac{i_0}{4^{n-1}},\frac{i_0+1}{4^{n-1}}\Big)\times
\Big[\frac{i_1}{2^{n-1}},\frac{i_1+1}{2^{n-1}}\Big)\times
 \cdots\times
\Big[\frac{i_d}{2^{n-1}},\frac{i_d+1}{2^{n-1}}\Big).\nonumber
\end{eqnarray}
Then we have
\begin{eqnarray}
|C'| = \mu(C')
&=&\frac{1}{2^{(d+1)(n-1)}}\int^{\frac{i_1+1}{2^{n-1}}}_{\frac{i_1}{2^{n-1}}}(x^1)^{\alpha}dx^1\nonumber\\
&=&\frac{1}{2^{(d+1)(n-1)}}\cdot
\frac{1}{\alpha+1}\Bigg[\left(\frac{i_1+1}{2^{n-1}}\right)^{\alpha+1}-\left(\frac{i_1}{2^{n-1}}\right)^{\alpha+1}\Bigg].\nonumber
\end{eqnarray}
Note that $C$ is one of $4\cdot 2^d$ cubes belonging to
$\mathbb{C}_{n}$ inside $C'$ and by the location of $C$  we have
either
\begin{eqnarray}
|C|=\frac{1}{2^{(d+1)n}}\cdot
\frac{1}{\alpha+1}\Bigg[\left(\frac{i_1+1}{2^{n-1}}\right)^{\alpha+1}-\left(\frac{i_1+1}{2^{n-1}}-\frac{1}{2^{n}}\right)^{\alpha+1}\Bigg]\label{eqn 5.23.1}
\end{eqnarray}
or
\begin{eqnarray}
|C|=\frac{1}{2^{(d+1)n}}\cdot
\frac{1}{\alpha+1}\Bigg[\left(\frac{i_1+1}{2^{n-1}}-\frac{1}{2^{n}}\right)^{\alpha+1}-\left(\frac{i_1}{2^{n-1}}\right)^{\alpha+1}\Bigg].\label{eqn 5.23.2}
\end{eqnarray}
{\bf Case 1}: Let $i_1\ge 1$ and $\alpha\geq 0$. Denoting
\begin{eqnarray}
a=\frac{i_1+1}{2^{n-1}},\quad b=\frac{i_1}{2^{n-1}},\quad
c=\frac{i_1+1}{2^{n-1}}-\frac{1}{2^{n}},\quad
\phi(x)=x^{\alpha+1},\nonumber
\end{eqnarray}
we get
\begin{eqnarray}
\frac{|C'|}{|C|}&=&2^{d+1}\cdot\frac{\phi(a)-\phi(b)}{\phi(a)-\phi(c)}\quad\textrm{or}\quad
2^{d+1}\cdot\frac{\phi(a)-\phi(b)}{\phi(c)-\phi(b)}\nonumber\\
&=&2^{d+1}\left(1+\frac{\phi(c)-\phi(b)}{\phi(a)-\phi(c)}\right)\quad\textrm{or}\quad 2^{d+1}\left(1+\frac{\phi(a)-\phi(c)}{\phi(c)-\phi(b)}\right)\nonumber\\
&=& 2^{d+1}\left(1+\frac{\phi'(\beta)}{\phi'(\alpha)}\right)
\quad\textrm{or}\quad
2^{d+1}\left(1+\frac{\phi'(\alpha)}{\phi'(\beta)}\right),
\end{eqnarray}
where $\alpha,\beta$ are some numbers satisfying
$b<\beta<c<\alpha<a$; we used mean value theorem. Since
$\alpha+1>1$, the function $\phi$
is convex and increasing on  $(0,\infty)$. Hence, we have
\begin{eqnarray}
\frac{\phi'(\beta)}{\phi'(\alpha)}\le 1,\quad\quad
\frac{\phi'(\alpha)}{\phi'(\beta)}\le
\frac{\phi'(a)}{\phi'(b)}=\frac{a^{\alpha}}{b^{\alpha}}=\left(\frac{i_1+1}{i_1}\right)^{\alpha}\le
2^{\alpha},\nonumber
\end{eqnarray}
and therefore
\begin{eqnarray}
\frac{|C'|}{|C|}\le 2^{d+1}(1+2^{\alpha})\le
2^{\alpha+d+2}.\nonumber
\end{eqnarray}
{\bf Case 2}:  Assume $i_1=0$ and $\alpha\geq 0$.  By similar but simpler calculation we obtain
\begin{eqnarray}
\frac{|C'|}{|C|}\le
2^{\alpha+d+2}.\nonumber
\end{eqnarray}
{\bf Case 3}: Assume $\alpha \in (-1,0)$.  If $|C|$ is given as in (\ref{eqn 5.23.2}), then
since $\phi(x)$ is concave,
$$
\frac{\left(\frac{i_1+1}{2^{n-1}}\right)^{\alpha+1}-\left(\frac{i_1}{2^{n-1}}\right)^{\alpha+1}}{\left(\frac{i_1+1}{2^{n-1}}-\frac{1}{2^{n}}\right)^{\alpha+1}-\left(\frac{i_1}{2^{n-1}}\right)^{\alpha+1}}\leq 2.
$$
Let $|C|$ be given as in (\ref{eqn 5.23.1}). If $i_1=0$, then
$$
\frac{\left(\frac{i_1+1}{2^{n-1}}\right)^{\alpha+1}-\left(\frac{i_1}{2^{n-1}}\right)^{\alpha+1}}{\left(\frac{i_1+1}{2^{n-1}}\right)^{\alpha+1}-\left(\frac{i_1+1}{2^{n-1}}-\frac{1}{2^{n}}\right)^{\alpha+1}}
=\frac{2^{\alpha+1}}{2^{\alpha+1}-1},
$$
and if $i_1 \geq 1$ then since $\phi$ is concave and $\phi'$ is positive on $(0,\infty)$
$$
\frac{\left(\frac{i_1+1}{2^{n-1}}\right)^{\alpha+1}-\left(\frac{i_1}{2^{n-1}}\right)^{\alpha+1}}{\left(\frac{i_1+1}{2^{n-1}}\right)^{\alpha+1}-\left(\frac{i_1+1}{2^{n-1}}-\frac{1}{2^{n}}\right)^{\alpha+1}}
\leq \frac{2^{-n+1}\phi'(\frac{i_1}{2^{n-1}})}{2^{-n}\phi'(\frac{i_1+1}{2^{n-1}})}\leq 2^{1-\alpha}.
$$
The lemma is proved.
\end{proof}

\begin{remark}
(i) By Lemma \ref{filtration1}, Lemma \ref{filtration2} and the outline of
Section 3.1, 3.2 of \cite{kr08} we get Lemma \ref{2010.03.05.1},
Theorem \ref{2010.03.05.2} and Theorem \ref{FS 1} below for free.

(ii) if $C_n\in \mathbb{C}_n$ and $C_m\in \mathbb{C}_m$ with $n\le
m$, then $C_n\cap C_m=C_m$ or $\emptyset$.

\end{remark}

\begin{definition}
We call $\tau=\tau(x)\in \mathbb{Z}\cup \{\infty\}$ a \emph{stopping
time} if $\{x:\tau(x)=n\}=\emptyset$ or union of some elements in
$\mathbb{C}_n$  for each $n\in \mathbb{Z}$.
\end{definition}
For $f\in L_{1,loc}(\Omega,\mu;\mathbb{R}^{d_1})$, $h\in L_{1,loc}(\bR^d_+,\nu,\bR^{d_1})$ and $n\in
\mathbb{Z}$ we define
\begin{eqnarray}
f_{|n}(t,x):=\frac{1}{\mu(C_n(t,x))}\int_{C_n(t,x)}f(s,y)\mu(dsdy)=\aint_{C_n(t,x)}f(s,y)\mu(dsdy)\nonumber,
\end{eqnarray}
\begin{eqnarray}
h_{|n}(x):=\frac{1}{\nu(D_n(t,x))}\int_{D_n(t,x)}h(y)\nu(dy)=\aint_{D_n(x)}h(y)\nu(dy)\nonumber,
\end{eqnarray}
and
\begin{eqnarray}
f_{|\tau}(t,x):= f_{|\tau(t,x)}(t,x)\quad\mathrm{if}\;\;\tau(t,x)\ne
\infty;\quad f_{|\tau}(t,x):= f(t,x)\quad\mathrm{if}\;\;\tau(t,x)=
\infty.\nonumber
\end{eqnarray}

\begin{lemma}\label{2010.03.05.1}
Let $\{\mathbb{C}_n\;:\; n\in \mathbb{Z}\}$ be a filtration of partitions
of $\bar\Omega$.

(i) Let $g\in L_{1,loc}(\Omega,\mu;\mathbb{R}^1)$, $g\ge 0$ and let
$\tau$ be a stopping time. Then
\begin{eqnarray}
\int_{\Omega} g_{|\tau}(t,x)
I_{\tau<\infty}(t,x)\mu(dtdx)&=&\int_{\Omega}
g(t,x) I_{\tau<\infty}(t,x)\mu(dtdx),\nonumber\\
\int_{\Omega} g_{|\tau}(t,x) \mu(dtdx)&=&\int_{\Omega} g(t,x)
\mu(dtdx).\nonumber
\end{eqnarray}

(ii) Let $g\in L_1(\Omega,\mu;\mathbb{R}^1)$, $g\ge 0$ and let
$\lambda>0$ be a constant. Then
\begin{eqnarray}
\tau(t,x):=\inf\{n: g_{|n}(t,x)>\lambda\}\quad\quad
(\inf\emptyset:=\infty)\nonumber
\end{eqnarray}
is a stopping time. Furthermore, we have
\begin{eqnarray}
0\le g_{|\tau}(t,x)I_{\tau<\infty}\le N_0\lambda,\quad
|\{(t,x):\tau(t,x)<\infty\}|\le
\lambda^{-1}\int_{\Omega}g(t,x)I_{\tau<\infty}\mu(dtdx).\nonumber
\end{eqnarray}
\end{lemma}

\begin{remark}
$($Riesz-Calder\'on-Zygmund decomposition$)$ Any $g\in
L_1(\Omega,\mu;\mathbb{R}^1)$ is decomposed by
\begin{eqnarray}
g=\xi+\eta,\nonumber
\end{eqnarray}
where $\xi=g-g_{|\tau}$,
$\eta=g_{|\tau}=g_{|\tau}\;I_{\tau<\infty}+g_{|\tau}\;I_{\tau=\infty}$.
Moreover, we have (i) $\eta\le N_0 \lambda$ a.e. (ii)
$|\{(t,x):\xi(t,x)\ne 0\}|\le
\lambda^{-1}\|g\|_{L_1(\Omega,\mu)}$ (iii) $\xi_{|\tau}=0$.
\end{remark}

\vspace{5mm}

Now, for $f\in L_{1,loc}(\Omega,\mu;\mathbb{R}^{d_1})$ we define the maximal function
\begin{eqnarray}
\mathcal{M}f(t,x):=\left(\sup_{n<\infty}|f^1|_{|n}(t,x),\;\ldots\;,\sup_{n<\infty}|f^{d_1}|_{|n}(t,x)\right)\nonumber
\end{eqnarray}
and the sharp function
\begin{eqnarray}
f^{\#}(t,x)=\left(\sup_{n<\infty}\aint_{C_n(t,x)}|f^1(s,y)-f^1_{|n}(s,y)|\mu(dsdy),\,\ldots,\,\sup_{n<\infty}\aint_{C_n(t,x)}|f^{d_1}(s,y)-f^{d_1}_{|n}(s,y)|\mu(dsdy)\right).\nonumber
\end{eqnarray}
We define $\cM h(x)$ and $h^{\#}(x)$ similarly for functions $h=h(x)\in L_{1,loc}(\bR^d_+,\nu;\mathbb{R}^{d_1})$.

\begin{theorem}\label{2010.03.05.2}
Let $p\in (1,\infty)$. Then for any $f\in L_p(\Omega,\mu;\mathbb{R}^{d_1})$ and $h\in L_p(\bR^d_+,\nu;\bR^{d_1})$, we
have
\begin{eqnarray}
\|\mathcal{M}f\|_{L_p(\Omega,\mu;\mathbb{R}^{d_1})}\le
N\|f\|_{L_p(\Omega,\mu;\mathbb{R}^{d_1})}, \quad \|\mathcal{M}h\|_{L_p(\bR^d_+,\nu;\mathbb{R}^{d_1})}\le
N\|h\|_{L_p(\bR^d_+,\nu;\mathbb{R}^{d_1})}\nonumber
\end{eqnarray}
where $N=N(\theta,p,d,d_1)$.
\end{theorem}

\begin{theorem}\label{FS 1}
Let $p\in (1,\infty)$. Then for any $f\in
L_p(\Omega,\mu;\mathbb{R}^{d_1})$  and $h\in L_p(\bR^d_+,\nu;\bR^{d_1})$ we have
\begin{eqnarray}
\|f\|_{L_p(\Omega,\mu;\mathbb{R}^{d_1})}\le
N\|f^{\#}\|_{L_p(\Omega,\mu;\mathbb{R}^{d_1})},  \quad   \|h\|_{L_p(\bR^d_+,\nu;\mathbb{R}^{d_1})}\le
N\|h^{\#}\|_{L_p(\bR^d_+,\nu;\mathbb{R}^{d_1})}\nonumber
\end{eqnarray}
where $N=N(\theta,p,d,d_1)$.
\end{theorem}

We investigate the relation between our maximal and sharp functions and more general ones. Let $B'_r(x')$ denote the open ball in $\bR^{d-1}$ of radius $r$ with center $x'$. For  $x=(x^1,x')\in \bR^d_+$ and $t\in \bR$, denote
$$
B_r(x)=B_r(x^1,x')=(x^1-r,x^1+r)\times B'_r(x'), \quad Q_r(t,x):=(t,t+r^2)\times B_r(x)
$$
and $\mathbb{Q}$ be the collection of all such open sets $Q_r(t,x)\subset \Omega$.
For $f\in L_{1,loc}(\Omega,\mu:\mathbb{R}^{d_1})$ we define
\begin{eqnarray}
f^i_Q=\aint_Q f^i\;d\mu,\quad \mathbb{M}f^i(t,x)=\sup_{(t,x)\in Q}\aint_Q
f^i d\mu,\quad (f^i)^{\sharp}(t,x)=\sup_{(t,x)\in Q}\aint_Q
|f^i-f^i_Q|d\mu,\quad i=1,\ldots,d_1,\nonumber
\end{eqnarray}
where the supremum is taken for all $Q\in \mathbb{Q}$ containing
$(t,x)$.  Denote
$$
\mathbb{M}f:=(\mathbb{M}f^1,\ldots,\mathbb{M}f^{d_1}),\quad
f^{\sharp}:=((f^1)^{\sharp},\ldots,(f^{d_1})^{\sharp}).
$$
For functions $h\in L_{1,loc}(\bR^d_+,\nu,\bR^{d_1})$, the functions $\bM h(x)$ and $(h)^{\sharp}(x)$ are defined similarly.

\begin{lemma}\label{lemma for FS}
For a scalar function $g=g(t,x)$ and $h=h(x)$ we have
\begin{eqnarray}
g^{\#}(t,x)\le N\; g^{\sharp}(t,x), \quad \quad h^{\#}(x)\le N\; h^{\sharp}(x) \nonumber
\end{eqnarray}
where $N=N(\theta,p,d)$.
\end{lemma}
\begin{proof}
We only prove the first assertion. For  $(t,x) \in \Omega$, denote the corresponding unique cube $C_n(t,x)\in
\mathbb{C}_n$ by
\begin{eqnarray}
\Big[\frac{i_0}{4^n},\frac{i_0+1}{4^n}\Big)\times
\Big[\frac{i_1}{2^n},\frac{i_1+1}{2^n}\Big)\times\cdots\times\Big[\frac{i_{d}}{2^n},\frac{i_{d}+1}{2^n}\Big)\nonumber
\end{eqnarray}
where $i_0,i_2,\ldots,i_{d}\in \mathbb{Z}$ and $i_1\in \{0\}\cup \mathbb{N}$. We
define $Q_{(n)}(t,x):=Q_{\frac{d}{2^n}}(t^*,x^*)$ with
$t^*=\frac{i_0}{4^n}$ and $x^*=(\frac{i_1+d}{2^n},\frac{i_2}{2^n},\ldots,\frac{i_d}{2^n})$. We have $(t,x)\in
C_n(t,x)\subset \overline{Q_{(n)}(t,x)}$ and
\begin{eqnarray}
\frac{|Q_{(n)}(t,x)|}{|C_n(t,x)|}=N(d)\cdot\frac{(i_1+2d)^{\alpha+1}-i_1^{\alpha+1}}{(i_1+1)^{\alpha+1}-i_1^{\alpha+1}}\label{ratio}
\end{eqnarray}
by simple calculation. If $i_1=0$, (\ref{ratio}) is $N(d)(2d)^{\alpha+1}$;
if $i_1\ge 1$ and $\alpha\geq 0$ then  (\ref{ratio}) is less than or equal to
\begin{eqnarray}
N(d)\cdot(2d)\left(\frac{i_1+2d}{i_1}\right)^{\alpha}\le N(d)\cdot(2d) \cdot
(1+2d)^{\alpha},\nonumber
\end{eqnarray}
by mean value theorem. If $\alpha\in (-1,0)$ then we use the concavity of $x^{\alpha+1}$ to prove that (2.4) is less then $N(d)(2d)^{\alpha+1}$.  The lemma is proved.
\end{proof}
Lemma \ref{lemma for FS} and Theorem \ref{FS 1} imply the following version of
Fefferman-Stein theorem:
\begin{theorem}\label{FS}
$($Fefferman-Stein$)$ Let $p\in (1,\infty)$. Then for any $f\in
L_p(\Omega,\mu;\mathbb{R}^{d_1})$ and $h\in L_p(\bR^d_+,\nu,\bR^{d_1})$,  we have
\begin{eqnarray}
\|f\|_{L_p(\Omega,\mu;\mathbb{R}^{d_1})}\le
N\|f^{\sharp}\|_{L_p(\Omega,\mu;\mathbb{R}^{d_1})}, \quad \quad \|h\|_{L_p(\bR^d_+,\nu;\mathbb{R}^{d_1})}\le
N\|h^{\sharp}\|_{L_p(\bR^d_+,\nu;\mathbb{R}^{d_1})} \nonumber
\end{eqnarray}
where $N=N(\theta,p,d,d_1)$.
\end{theorem}

 The following lemma will be used in the proof of   Theorem \ref{HL} below.
\begin{lemma}\label{key lemma}
 Let $\alpha>-1$ and $\phi(x)=x^{\alpha+1}$ on $x>0$. Then for any
$x>0$ and $r>0$ we have
\begin{eqnarray}
\frac{\phi(x+2r)-\phi(x+r)}{\phi(x+r)-\phi(x)}\;\le
\;2^{\alpha+1}.\nonumber
\end{eqnarray}
\end{lemma}

\begin{proof}
If $\alpha\in (-1,0]$ the claim is obvious since $\phi$ is concave.

Assume $\alpha>0$, fix  $r>0$, and define
\begin{eqnarray}
f(x):=\frac{\phi(x+2r)-\phi(x+r)}{\phi(x+r)-\phi(x)}.\nonumber
\end{eqnarray}
We show that $f'(x)\le 0$ for $x>0$ so that $f(x)\le
f(0)=2^{\alpha+1}-1$; note that $f(0)$ does not depend on $r$. A
simple calculation shows
\begin{eqnarray}
f'(x)=r(\alpha+1)\cdot \frac{2(x+2r)^{\alpha}x^{\alpha}-(x+2r)^{\alpha}(x+r)^{\alpha}-(x+r)^{\alpha}x^{\alpha}}
{((x+r)^{\alpha+1}-x^{\alpha+1})^2}.\label{2010.03.09.1}
\end{eqnarray}
The numerator in (\ref{2010.03.09.1}) is
\begin{eqnarray}
2\cdot x^{\alpha}(x+r)^{\alpha}(x+2r)^{\alpha}
\cdot\left[(x+r)^{-\alpha}-\frac{x^{-\alpha}+(x+2r)^{-\alpha}}{2}\right].\label{2010.03.09.2}
\end{eqnarray}
Since the function $x^{-\alpha}$ is convex and $x+r$ is the
midpoint of $x$ and $x+2r$, the square bracket in
(\ref{2010.03.09.2}) is non-positive and so is $f'(x)$. The lemma
is proved.
\end{proof}

\begin{theorem}\label{HL}
$($Hardy-Littlewood$)$ Let $p\in (1,\infty)$. Then for $f\in
L_p(\Omega,\mu;\mathbb{R}^{d_1})$  and $h\in L_p(\bR^d_+,\nu,\bR^{d_1})$ we have
\begin{eqnarray}
\|\mathbb{M}f\|_{L_p(\Omega,\mu;\mathbb{R}^{d_1})}\le N
\|f\|_{L_p(\Omega,\mu;\mathbb{R}^{d_1})}, \quad  \|\mathbb{M}h\|_{L_p(\bR^d_+,\nu;\mathbb{R}^{d_1})}\le N
\|h\|_{L_p(\bR^d_+,\nu;\mathbb{R}^{d_1})}.  \nonumber
\end{eqnarray}
\end{theorem}
\begin{proof} Again we only proof the first assertion. We follow the outline for the proof of Theorem 3.3.2
which does not involve a weight in the norm. Without loss of generality we assume $d_1=1$ and $g:=f\ge 0$.

For $\lambda>0$, denote $A(\lambda):=\{(t,x) : \mathbb{M}g(t,x)>\lambda\}$. Then since $\mathbb{M}g$ is  lower semi-continuous, $A(\lambda)$ is open.
To prove the theorem it is enough to show that for any $\lambda>0$ and compact set $K\subset A(\lambda)$
\begin{eqnarray}
|K|\le
\frac{N}{\lambda}\int_{\Omega}I_{A(\lambda)}(t,x)g(t,x)\mu(dtdx),\nonumber
\end{eqnarray}
where $N=N(\theta,p,d)$. For the details see the proof of Theorem 3.3.2 of \cite{kr08}.

 For any $(t,x)\in K$ there exists $Q$ containing
$(t,x)$ such that $\int_Q g d\mu > \lambda |Q|$. Also, we observe
that $Q\subset A(\lambda)$ and there exists a finite cover $\{Q_1,\ldots,Q_n\}$ of $K$
 such that
\begin{eqnarray}
\int_{Q_i} g d\mu>\lambda |Q_i|.\nonumber
\end{eqnarray}
For $Q=(t-\frac12 r^2,t+\frac12 r^2)\times(x^1-r,x^1+r)\times B'_r(x') \in \mathbb{Q}$, denote
$$3Q:=\big(t-\frac32 r^2,t+\frac32 r^2\big)\times(x^1-3r,x^1+3r)\times B'_{3r}(x').$$
When $Q$
is close to the boundary of $\Omega$, $3Q$ may not be in $\Omega$.
Hence, we define
\begin{eqnarray}
Q^*=3Q \cap \Omega.\nonumber
\end{eqnarray}
Using a Vitali covering argument  one can find the disjoint subset $\{\tilde Q_1,\ldots,\tilde Q_k\}$ of $
\{Q_1,\ldots, Q_n\}$ satisfying $K\subset \bigcup^k_{j=1}\tilde{Q_j}^*$ (see the proof of Theorem 3.3.2 of \cite{kr08}). To measure $|K|$ we compute the ratio $\frac{|\tilde{Q_j}^*|}{|\tilde{Q_j}|}$.
For $Q_j=(t-\frac{r^2}{2},t+\frac{r^2}{2})\times (x^1-r,x^1+r)\times B'_r(x')$ we have
\begin{eqnarray}
\frac{|\tilde{Q_j}^*|}{|\tilde{Q_j}|}=3^d\cdot\frac{\phi(x+3r)-\phi((x-3r)\vee
0)}{\phi(x+r)-\phi(x-r)},\nonumber
\end{eqnarray}
where $\phi(x)=x^{\theta-d+p+1}$ and $a\vee b:=\max\{a,b\}$. We note
\begin{eqnarray}
&&\frac{\phi(x+3r)-\phi((x-3r)\vee
0)}{\phi(x+r)-\phi(x-r)}\nonumber\\
&=&\frac{\phi(x-r)-\phi((x-3r)\vee
0)+\phi(x+r)-\phi(x-r)+\phi(x+3r)-\phi(x+r)}{\phi(x+r)-\phi(x-r)}\nonumber\\
&\le&
2+\frac{\phi(x+3r)-\phi(x+r)}{\phi(x+r)-\phi(x-r)},\label{2010.03.09.4}
\end{eqnarray}
where the last inequality is true since $\phi$ is increasing and
convex.  Now, Lemma \ref{key lemma} with
$x-r$, $2r$ instead of $x$, $r$ implies (\ref{2010.03.09.4}) is less
than or equal to $2+2^{\alpha+1}$. Hence, we have
\begin{eqnarray}
\frac{|\tilde{Q_j}^*|}{|\tilde{Q_j}|}\le
3^d\cdot(2+2^{\alpha+1}),\quad |\tilde{Q_j}^*|\le
3^d\cdot(2+2^{\alpha+1})|\tilde{Q_j}|.\nonumber
\end{eqnarray}
Thus,
\begin{eqnarray}
|K|&\le& \sum^k_{j=1}|\tilde{Q_j}^*|\le
3^d\cdot(2+2^{\alpha+1})\sum^k_{j=1}|\tilde{Q_j}|\nonumber\\
&\le&
3^d\cdot(2+2^{\alpha+1})\lambda^{-1}\sum^k_{j=1}\int_{\tilde{Q_j}}g\,d\mu\le
3^d\cdot(2+2^{\alpha+1})\lambda^{-1}\int_{\Omega}gI_{A(\lambda)}\,d\mu.\nonumber
\end{eqnarray}
The theorem is proved.
\end{proof}


\mysection{A weighted $L_p$-theory for  systems in a half space}
                                                    \label{main result}

Let $C^{\infty}_0(\mathbb{R}^d;\mathbb{R}^{d_1})$ denote
the set of all $\mathbb{R}^{d_1}$-valued infinitely differentiable
functions with compact support in $\mathbb{R}^d$. By $\mathcal{D}$ we denote  the
space of $d$-dimensional distributions on $C^{\infty}_0(\mathbb{R}^d;\mathbb{R}^{d_1})$;
precisely, for $u\in \mathcal{D}$ and $\phi\in C^{\infty}_0(\mathbb{R}^d;\mathbb{R}^{d_1})$ we
define $(u,\phi)\in \mathbb{R}^{d_1}$ with components
$(u,\phi)^k=(u^k,\phi^k)$, $k=1,\ldots,d_1$; each $u^k$ is a usual
scalar-valued distribution.

For $p\in (1,\infty)$ we define $L_p=L_p(\mathbb{R}^d;\mathbb{R}^{d_1})$ as the space of all
$\mathbb{R}^{d_1}$-valued functions $u=(u^1,\ldots,u^{d_1})$ satisfying
\[
\|u\|^p_{L_p}:=\sum^{d_1}_{k=1}\|u^k\|^p_{L_p}<\infty.
\]
Denote $x=(x^1,\ldots,x^d)$.
In this paper we define
\[
\|u_x\|^p_{L_p}=\sum_{i=1}^d \|u_{x^i}\|^p_{L_p},\quad
\|u_{xx}\|^p_{L_p}=\sum_{i,j=1}^d \|u_{x^ix^j}\|^p_{L_p},\quad\textrm{etc.}
\]

For any $\gamma\in \bR$,  define the
space of Bessel potential
$H^{\gamma}_p=H^{\gamma}_p(\mathbb{R};\mathbb{R}^{d_1})$ as the space
of all distributions $u$ on $\bR^d$ such that $(1-\Delta)^{\gamma/2}u\in L_p$, where
 each component is defined by
\[
((1-\Delta)^{\gamma/2}u)^k=(1-\Delta)^{\ga/2}u^k
\]
and the norm is given by
\[
\|u\|_{H^{\gamma}_p}:=\|(1-\Delta)^{\ga/2}u\|_{L_p}.
\]
Then $H^{\gamma}_p$ is a Banach space with the given norm and
$C^{\infty}_0(\mathbb{R}^d;\mathbb{R}^{d_1})$ is dense in $H^{\gamma}_p$ (see \cite{T}). Note that $H^{\ga}_p$ are
usual Sobolev spaces for $\ga=0,1,2,\ldots$. It is well known that
the first order differentiation operator, $D:H^{\gamma}_{p}\to
H^{\gamma-1}_p$, is bounded. On the other hand, if $\text{supp}\,
(u) \subset (a,b)$, where $-\infty<a<b<\infty$, then
\begin{equation}
                                        \label{eqn 5.1.1}
\|u\|_{H^{\gamma}_p}\leq c(d,a,b)\|u_x\|_{H^{\gamma-1}_p}
\end{equation}
(see, for instance, Remark 1.13 in \cite{kr99}).

Now we introduce the weighted Sobolev spaces taken from
 \cite{kr99} and
\cite{Lo2}. Take a nonnegative real-valued function $\zeta(x)=\zeta(x^1)\in
C^{\infty}_0(\bR_+)$ such that
\begin{equation}
                                       \label{eqn 5.6.5}
\sum_{n=-\infty}^{\infty}\zeta(e^{n+s})>c>0, \quad \forall s\in \bR,
\end{equation}
where $c$ is a constant. Note that  any nonnegative function $\zeta$
with $\zeta>0$ on $[1,e]$ satisfies (\ref{eqn 5.6.5}). For
$\theta\in \bR$,  let
$H^{\gamma}_{p,\theta}:=H^{\gamma}_{p,\theta}(\mathbb{R}^d_+;\mathbb{R}^{d_1})$
denote the set of all $d$-dimensional distributions $u=(u^1,u^2,\cdots u^{d_1})$  on $\bR^d_+$
such that
\begin{equation}
                  \label{def weight}
\|u\|^p_{H^{\gamma}_{p,\theta}}:= \sum_{n\in\bZ} e^{n\theta}
\|\zeta(\cdot) u(e^{n} \cdot)\|^p_{H^{\gamma}_p}<\infty.
\end{equation}
It is known that for different $\zeta$ satisfying (\ref{eqn 5.6.5}),
we get the same spaces $H^{\gamma}_{p,\theta}$ with equivalent
norms, and for any $\eta\in C^{\infty}_0(\bR_+;\bR)$,
\begin{equation}
                            \label{eqn 5.6.1}
\sum_{n=-\infty}^{\infty}
e^{n\theta}\|\eta(\cdot)u(e^n\cdot)\|^p_{H^{\ga}_p} \leq N
\sum_{n=-\infty}^{\infty}
e^{n\theta}\|\zeta(\cdot)u(e^n\cdot)\|^p_{H^{\ga}_p},
\end{equation}
where $N$ depends only on $\gamma,\theta,p,d,d_1,\eta,\zeta$.
Furthermore,
 if $\gamma$ is a nonnegative integer, then
\begin{eqnarray}
\|u\|^p_{H^{\gamma}_{p,\theta}}\sim \sum_{|\beta|\le\gamma}
\int_{\bR_+^d}|(x^1)^{|\beta|} D^{\beta}u(x)|^p(x^1)^{\theta-d} \,dx.
\label{non-negative
integer}
\end{eqnarray}

 Let $M^{\alpha}$ be the operator of
multiplying by $(x^1)^{\alpha}$ and  $M:=M^1$. For $\nu\in (0,1]$, denote
$$
|u|_{C}=\sup_{x\in\bR^d_+}|u(x)|, \quad [u]_{C^{\nu}}=\sup_{x\neq
y}\frac{|u(x)-u(y)|}{|x-y|^{\nu}}.
$$
Below we collect some other important properties of the spaces
$H^{\gamma}_{p,\theta}$.

\begin{lemma} $(\cite{kr99}, \cite{kr99-1})$
                  \label{lemma 1}
                  Let $\gamma, \theta\in \bR$ and $p\in (1,\infty)$.
\begin{itemize}
\item[(i)] $C^{\infty}_0(\mathbb{R}^d_+;\mathbb{R}^{d_1})$ is dense in $H^{\gamma}_{p,\theta}$.

\item[(ii)] Assume that $\gamma=m+\nu+d/p$ for some $m=0,1,\cdots$ and
$\nu\in (0,1]$.  Then for any $u\in H^{\gamma}_{p,\theta}$ $i\in
\{0,1,\cdots,m\}$, we have
\begin{equation}
                         \label{eqn 3.31.4}
|M^{i+\theta/p}D^iu|_{C}+[M^{m+\nu+\theta/p}D^m u]_{C^{\nu}}\leq N
\|u\|_{ H^{\gamma}_{p,\theta}}.
\end{equation}

\item[(iii)] Let $\alpha\in \bR$. Then
$M^{\alpha}H^{\gamma}_{p,\theta+\alpha p}=H^{\gamma}_{p,\theta}$ and
$$
\|u\|_{H^{\gamma}_{p,\theta}}\leq N
\|M^{-\alpha}u\|_{H^{\gamma}_{p,\theta+\alpha p}}\leq
N\|u\|_{H^{\gamma}_{p,\theta}}.
$$


\item[(iv)] For any $MD, DM: H^{\gamma}_{p,\theta}\to H^{\gamma-1}_{p,\theta}$ are
bounded linear operators, and
$$
\|u\|_{H^{\gamma}_{p,\theta}}\leq N\|u\|_{H^{\gamma-1}_{p,\theta}}+N
\|Mu_x\|_{H^{\gamma-1}_{p,\theta}}\leq N
\|u\|_{H^{\gamma}_{p,\theta}},
$$
$$
\|u\|_{H^{\gamma}_{p,\theta}}\leq N\|u\|_{H^{\gamma-1}_{p,\theta}}+N
\|(Mu)_x\|_{H^{\gamma-1}_{p,\theta}}\leq N
\|u\|_{H^{\gamma}_{p,\theta}}.
$$
Furthermore, if $\theta\neq d-1, d-1+p$, then
\begin{equation}
                        \label{eqn 4.25.3}
\|u\|_{H^{\gamma}_{p,\theta}}\leq N
\|Mu_x\|_{H^{\gamma-1}_{p,\theta}}, \quad  \|u\|_{H^{\gamma}_{p,\theta}}\leq N
\|(Mu)_x\|_{H^{\gamma-1}_{p,\theta}}.
\end{equation}

\item[(v)] For $i=0,1$\; let $\kappa\in [0,1],\; p_i\in (1,\infty),\; \gamma_i,\;\theta_i \in \bR$ and assume the relations
\[
\gamma=\kappa \gamma_1 +(1-\kappa)\gamma_0,\quad \frac1p=\frac{\kappa}{p_1}+ \frac{1-\kappa}{p_0},\quad
\frac{\theta}{p}=\frac{\theta_1\kappa}{p_1}+\frac{\theta_0(1-\kappa)}{p_0}.
\]
Then
$$
\|u\|_{H^{\gamma}_{p,\theta}}\leq N\|u\|^{\kappa}_{H^{\gamma_1}_{p_1,\theta_1}}\|u\|^{1-\kappa}_{H^{\gamma_0}_{p_0,\theta_0}}.
$$

\end{itemize}
\end{lemma}
\begin{remark}
                \label{remark last}
 Let  $\theta\in  (d-1, d-1+p)$ and $n$ be a nonnegative integer. By Lemma \ref{lemma 1} $(iii), (iv)$
\begin{equation}
\|M^{-n}v\|_{H^{\gamma}_{p,\theta}}\le N \|D^n v\|_{H^{\gamma-n}_{p,\theta}}\label{2011.03.21.1}
\end{equation}
 for any $v\in C^{\infty}_0(\mathbb{R}^d;\mathbb{R}^{d_1})$. Indeed, since $\theta+mp\ne d-1,d-1+p$ for any integer $m$
\begin{eqnarray}
\|M^{-n}v\|_{H^{\gamma}_{p,\theta}}&\le& N \|M^{-1}v\|_{H^{\gamma}_{p,\theta-(n-1)p}}\le N \|v_x\|_{H^{\gamma-1}_{p,\theta-(n-1)p}}\nonumber\\
&\le& N\|M^{-1}v_x\|_{H^{\gamma-1}_{p,\theta-(n-2)p}}\le N \|D^2v\|_{H^{\gamma-2}_{p,\theta-(n-2)p}}\ldots.\nonumber
\end{eqnarray}
\end{remark}

\vspace{3mm}

For $-\infty \leq S<T\leq \infty$, we define the Banach spaces:
$$
\bH^{\gamma}_{p,\theta}(S,T):=L_p((S,T),H^{\gamma}_{p,\theta}), \; \bH^{\gamma}_{p,\theta}(T):=\bH^{\gamma}_{p,\theta}(0,T),\;
\bL_{p,\theta}(S,T):=H^0_{p,\theta}(S,T),\; \bL^{\gamma}_{p,\theta}(T):=\bL^{\gamma}_{p,\theta}(0,T)
$$
with norms given by
\[
\|u\|^p_{\bH^{\ga}_{p,\theta}(S,T)}=\int^{T}_S\|u(t)\|^p_{H^{\gamma}_{p,\theta}}dt.
\]

\begin{lemma}
               \label{duality}
For $\phi, \psi \in C^{\infty}_0((S,T)\times \bR^d_+)$, define $(\phi,\psi)=\int_{S}^T\int_{\bR^d_+}\phi(s,x)\psi(t,x)dtdx$.
For $p\in (1,\infty)$ and $\gamma,\theta\in\bR$, define $\gamma',p',\theta'$ so that
\[
\gamma'=-\gamma,\quad \frac1p+\frac{1}{p'}=1,\quad \frac{\theta}{p}+\frac{\theta'}{p'}=d.
\]
Then for any $\phi\in C^{\infty}_0((S,T)\times \bR^d_+)$
\begin{eqnarray}
\|\phi\|_{\bH^{\gamma}_{p,\theta}(S,T)}\le N\;\sup_{\psi\in C^{\infty}_0((S,T)\times \bR^d_+)}\frac{(\phi,\psi)}{\|\psi\|_{\bH^{\gamma'}_{p',\theta'}(S,T)}}
\le N \|\phi\|_{\bH^{\gamma}_{p,\theta}(S,T)},\nonumber
\end{eqnarray}
where $N$ is independent of $\phi$. Moreover the relation $(\phi,\psi)$ can be extended  by continuity on all $\phi\in \bH^{\gamma}_{p,\theta}(S,T)$ and $\psi\in \bH^{\gamma'}_{p',\theta'}(S,T)$, and then
it identifies the dual to $\bH^{\gamma}_{p,\theta}(S,T)$ with $\bH^{\gamma'}_{p',\theta'}(S,T)$.
\end{lemma}
\begin{proof}
See Theorem 2.5 of \cite{kr99-1}; this actually proves the duality between $H^{\gamma}_{p,\theta}$ and $H^{\gamma'}_{p',\theta'}$,
but the proof of our claim is essentially the same. The only difference is that one has to consider integrations on the time variable, too.
\end{proof}

Finally, we set
$U^{\gamma}_{p,\theta}:=M^{1-2/p}H^{\gamma-2/p}_{p,\theta}$, meaning
that any $u \in U^{\gamma}_{p,\theta}$ has the form
$u=M^{1-2/p}\cdot v$ with $v \in H^{\gamma-2/p}_{p,\theta}$ and
$\|u\|_{U^{\gamma}_{p,\theta}}:=\|M^{-1+2/p}u\|_{H^{\gamma-2/p}_{p,\theta}}=\|v\|_{H^{\gamma-2/p}_{p,\theta}}$.
Using these spaces, we define our solution spaces.

\begin{definition}\label{md}
We write $u\in \frH^{\gamma+2}_{p,\theta}(S,T)$ if $u\in
M\bH^{\gamma+2}_{p,\theta}(S,T)$, $u(S,\cdot)\in U^{\gamma+2}_{p,\theta}$ ($u(-\infty,\cdot):=0$ if $S=-\infty$),
and for some $\tilde f\in M^{-1}\bH^{\gamma}_{p,\theta}(T)$  it holds $u_t=\tilde{f}$ in the sense of distributions, that is
 for any $\phi\in C^{\infty}_0(\mathbb{R}^d;\mathbb{R}^{d_1})$ the equality
\begin{equation}\label{e}
(u(t,\cdot),\phi)= (u(S,\cdot),\phi)+ \int^t_S( \tilde
f(s,\cdot),\phi)ds
\end{equation}
holds for all $t\in (S,T)$. In this case we write $u_t=\tilde{f}$.
The norm in $\frH^{\gamma+2}_{p,\theta}(S,T)$ is defined by
$$
\|u\|_{\frH^{\gamma+2}_{p,\theta}(S,T)}=\|M^{-1}u\|_{\bH^{\gamma+2}_{p,\theta}(S,T)}+\|Mu_t\|_{\bH^{\gamma}_{p,\theta}(S,T)} +\|u(S,\cdot)\|_{U^{\gamma+2}_{p,\theta}}.
$$
Define $\frH^{\gamma+2}_{p,\theta}(T):=\frH^{\gamma+2}_{p,\theta}(0,T)$ and $\frH^{\gamma+2}_{p,\theta}:=\frH^{\gamma+2}_{p,\theta}(0,\infty)$.
\end{definition}

\begin{theorem}
                      \label{banach}
(i) The space $\frH^{\gamma+2}_{p,\theta}(S,T)$ is a Banach space.

(ii) Let $0<T<\infty$. Then   for any $u\in \frH^{\gamma+2}_{p,\theta}(T)$,
$$
\sup_{t\leq T}\|u(t)\|_{\bH^{\gamma+1}_{p,\theta}}\leq N(d,p,\theta,T)\|u\|_{\frH^{\gamma+2}_{p,\theta}(T)}.
$$

(iii) Let $0<T<\infty$. For any nonnegative integer $n \geq \gamma+2$, the set
$$
\frH^{\gamma+2}_{p,\theta}(T) \bigcap \bigcup_{k=1}^{\infty} C([0,T],C^n_0(G_k))
$$
where $G_k=(1/k,k)\times \{|x'|<k\}$ is dense in $\frH^{\gamma+2}_{p,\theta}(T)$.
\end{theorem}
\begin{proof}
See Theorem 2.9 and Theorem 2.11 of \cite{KL2}.
\end{proof}

Here are  some interior H\"older estimates of functions in the space $\frH^{\gamma+2}_{p,\theta}(T)$.
\begin{theorem}
                      \label{thm interior}
  Let $p>2$ and assume
  $$
  2/p<\alpha<\beta\leq 1, \quad \gamma+2-\beta-d/p=k+\varepsilon,
  $$
  where $k\in \{0,1,2,\cdots\}$ and $\varepsilon \in (0,1]$. Denote $\sigma=\beta-1+\theta/p$. Then for any $u\in \frH^{\gamma+2}_{p,\theta}(T)$ and multi-indices $i,j$ such that
  $|i|\leq j$ and $|j|=k$,

  (i) the functions $D^iu(t,x)$ are continuous in $[0,T]\times \bR^d_+$ and
  $$
  M^{\sigma+|i|}D^iu(t,\cdot)- M^{\sigma+|i|}D^iu(0,\cdot)   \in C^{\alpha/2-1/p}([0,T], C(\bR^d_+));
  $$

  (ii) there exists a constant $N=N(p,d,\alpha,\beta)$ so that
  \begin{eqnarray}
                &&\sup_{t,s\leq T}    \left(\frac{\big|M^{\sigma+|i|}D^i(u(t)-u(s))\big|_{C(\bR^d_+)}}{|t-s|^{\alpha/2-1/p}}  +
  \frac{\big[M^{\sigma+|j|+\varepsilon}D^j(u(t)-u(s))\big]_{C^{\varepsilon}}}{|t-s|^{\alpha/2-1/p}}   \right)\nonumber\\
  &\leq& NT^{(\beta-\alpha)/2} \|u\|_{\frH^{\gamma+2}_{p,\theta}(T)}. \label{eqn 4.24.1}
  \end{eqnarray}
  \end{theorem}
  \begin{proof}
  See Theorem 4.7 of \cite{Kr01}.
  \end{proof}

  \begin{remark} (see Remark 4.8 of \cite{Kr01} for details) For instance, if  $\theta=d$, $\gamma\geq -1$ and $\kappa_0=1-\frac{2}{p}-\frac{d}{p}>0$, then     for any $\kappa \in (0,\kappa_0)$ and
  $u\in \frH^1_{p,\theta}(T)$ with $u(0)=0$,
  \begin{equation}
                          \label{eqn 4.24.4}
  \sup_{t\leq T}\sup_{x,y\in \bR^d_+}\frac{|u(t,x)-u(t,y)|}{|x-y|^{\kappa}}<\infty.
  \end{equation}
  \begin{equation}
                          \label{eqn 4.24.5}
   \sup_{x\in \bR^d_+}\sup_{s,t\leq T}\frac{|u(t,x)-u(s,x)|}{|t-s|^{\kappa/2}}<\infty.
  \end{equation}
  Indeed, for (\ref{eqn 4.24.4}) take $j=0, \beta=\kappa_0-\kappa+2/p$ and $\varepsilon=1-\beta-d/p=\kappa=-\sigma$, then $\sigma+|j|+\varepsilon=0$ and (\ref{eqn 4.24.1}) yields (\ref{eqn 4.24.4}). Also for
  (\ref{eqn 4.24.5}), take $i=0, \alpha=\kappa+2/p, \beta=1-d/p$ then $\sigma+|i|=0$, $2/p<\alpha<\beta<1$ and $\alpha/2-1/p=\kappa/2$.
  \end{remark}

  \vspace{3mm}

For any $d_1\times d_1$ matrix $C=(c_{kr})$ we let
$$
|C|:=\sqrt{\sum_{k,r}(c_{kr})^2}.
$$
We set $A^{ij}=(a^{ij}_{kr})_{k,r=1,\ldots,d_1}$ for each $i,j=1,\ldots,d$.  Throughout the article we assume the followings.
\begin{assumption}
                     \label{main assumption2}
For each $i$ and $j$,  $A^{ij}$ depends only on $t$ and there exist finite constants
$\delta,K>0$ so that
\begin{equation}
                    \label{assumption 1}
\delta|\xi|^2\leq \sum_{i,j=1}^d(\xi^i)^*A^{ij} \;\xi^i
\end{equation}
for all (real valued) $d_1\times d$-matrix $\xi$, where $\xi^i$ denotes the $i$-th column of $\xi$.
Also, there exists a constant $K<\infty$ such that
\begin{equation}\label{assumption 2}
\left|A^{ij}\right|\le K, \quad\forall\;\;i,j=1,\ldots,d,
\end{equation}
where $*$ means matrix transposition.
\end{assumption}

 We recall (\ref{main system}) and write it as
\begin{equation}
                    \label{eqn system0}
   u^k_t=a^{ij}_{kr}(t)u^r_{x^ix^j}+f^k, \quad u^k(S)=u^k_0, \quad \quad k=1,2,\cdots, d_1,
\end{equation}
assuming  the summation convention  on  indices $i,j, r$; such
convention will be used throughout the article. In short,  we will write (\ref{eqn system0})  as
\begin{equation}
                       \label{eqn system}
u_t=A^{ij}(t)u_{x^ix^j}+f, \quad u(S)=u_0,
\end{equation}
where we regard $u,u_0,f$ as $d_1\times 1$ matrix-valued functions.

\begin{definition}
                       \label{definition 5.8.1}
A $d$-dimensional distribution-valued function $u$ defined
on $(S,T)$ is a solution of (\ref{eqn system}) in
$\frH^{\gamma+2}_{p,\theta}(S,T)$ if $u\in
M\bH^{\gamma+2}_{p,\theta}(S,T)$, $u(S)\in U^{\gamma+2}_{p,\theta}$ ($u(-\infty,\cdot):=0$ if $S=-\infty$) and (\ref{eqn system}) holds in the
sense of distributions, that is (\ref{e}) holds  with $\tilde f=A^{ij}u_{x^ix^j}+f$.
\end{definition}

The following is  our $L_p$-theory for  the parabolic  system (\ref{eqn system}). The  proof is given in section \ref{section proof}.

\begin{theorem}
                                       \label{main theorem}
Let $p\in (1,\infty)$ and $\gamma\geq 0$. Assume $\theta\in (d+1-p, d+p-1)$ if $p\in (1,2]$ and $\theta\in (d-1,d+1)$ if $p\in (2,\infty)$.
 Then for any  $f\in M^{-1}\bH^{\gamma}_{p,\theta}(T)$ and $u_0\in
U^{\gamma+2}_{p,\theta}$  system (\ref{eqn system}) admits a
unique solution $u\in \frH^{\gamma+2}_{p,\theta}(T)$, and for this solution we have
\begin{equation}
                        \label{a priori}
\|u\|_{\frH^{\gamma+2}_{p,\theta}(T)}\leq
N\left(\|Mf\|_{\bH^{\gamma}_{p,\theta}(T)}+\|u_0\|_{U^{\gamma+2}_{p,\theta}}\right),
\end{equation}
where $N=N(\gamma,p,\theta,\delta,K)$.
\end{theorem}

\begin{remark}
Various interior H\"older estimates of the solution in Theorem \ref{main theorem} can be obtained according to Theorem \ref{thm interior}. Also see Lemma \ref{lemma 3} and Lemma \ref{101004.14.53}.
\end{remark}

\begin{remark}
(i) The proof of Theorem \ref{main theorem} is based on a sharp function estimate (Lemma \ref{lemma 06.08.1}). If  $d_1=1$, then Lemma \ref{lemma 06.08.1} can be proved for any $\theta\in (d-1,d-1+p)$ as long as $p>1$; we will prove this in a subsequent article for parabolic equations with
(weighted) BMO-coefficients.

\vspace{2mm}

(ii) It is known (see Remark 3.6 of \cite{KL2}) that if $\theta\not \in (d-1,d-1+p)$, then  Theorem \ref{main theorem} is not true even for the heat equation $u_t=\Delta u+f$.

\end{remark}

Now we present our $L_p$-theory for the elliptic system (\ref{main system elliptic}). The proof is given in section \ref{section proof}.

\begin{theorem}
                                       \label{main theorem-elliptic}
Let $p\in (1,\infty)$,  $\gamma\geq 0$ and $A^{ij}$ be independent of $t$. Assume $\theta\in (d+1-p, d+p-1)$ if $p\in (1,2]$ and $\theta\in (d-1,d+1)$ if $p\in (2,\infty)$.
 Then for any  $f=(f^1,f^2,\cdots,f^{d_1})\in M^{-1}H^{\gamma}_{p,\theta}$  the system (\ref{main system elliptic}) admits a
unique solution $u\in MH^{\gamma+2}_{p,\theta}$, and for this solution we have
$$
\|M^{-1}u\|_{H^{\gamma+2}_{p,\theta}}\leq
N\|Mf\|_{H^{\gamma}_{p,\theta}},
$$
where $N=N(\gamma,p,\theta,\delta,K)$.
\end{theorem}

\begin{remark}
Theorem \ref{main theorem}  and Theorem \ref{main theorem-elliptic} hold not only for $\gamma\geq 0$ but also for any $\gamma<0$. This can be easily proved by using the results for $\gamma\geq 0$ and repeating the arguments used for single equations (see the proof of Theorem 5.6 of \cite{kr99}).
\end{remark}

\mysection{Preliminary estimates : Some local estimates  of solutions}
                            \label{local estimate}

In this section we prove a version of Theorem \ref{main theorem} for $\theta=d$. This result is used to derive
some local estimates of $D^{\alpha}u$ for any multi-index $\alpha$,   where $u$ is a solution of (\ref{eqn system}).

First, we introduce some results for systems defined on the  {\bf{entire space}}.
 For $-\infty\leq S<T\leq \infty$ we denote $\bH^{\ga}_p(S,T):=L_p((S,T),H^{\ga}_p)$ and $\bH^{\ga}_p(T):=\bH^{\ga}_p(0,T)$.

\begin{theorem}
                     \label{thm entire}
Let $\ga\in \bR$ and $-\infty\leq S<T \leq \infty$. Let $f\in \bH^{\ga}_p(S,T)$ and  $u\in \bH^{\ga+2}_p(S,T)$ satisfy
$$
u_t=A^{ij}(t)u_{x^ix^j}+f, \quad t>S, x\in \bR^d.
$$
Additionally assume $u(S,\cdot)=0$ if $S>-\infty$. Then
\begin{equation}
                            \label{eqn single}
\|u_{xx}\|^p_{\bH^{\ga}_p(S,T)}\leq N(d,p,\delta,K)\|f\|^p_{\bH^{\ga}_p(S,T)}.
\end{equation}
Also if $-\infty<S<T<\infty$, then
$$
\|u\|^p_{\bH^{\ga+2}_p(S,T)}\leq N(d,p,\delta,K,S,T)\|f\|^p_{\bH^{\ga}_p(S,T)}.
$$
\end{theorem}
\begin{proof}
This is a classical result. See, for instance, Theorem 1.1 of \cite{Lee}. Actually in \cite{Lee} the theorem is proved only when $\ga=0$, but the general case follows by the fact
the operator $(1-\Delta)^{\mu/2}: \bH^{\ga}_p(S,T) \to \bH^{\ga-\mu}_p(S,T)$ is an isometry.
\end{proof}

Theorem \ref{eqn single} yields the following result.
\begin{corollary}
Let $u\in C^{\infty}_0(\bR^{d+1};\bR^{d_1})$. Then
\begin{equation}
                            \label{eqn single-3}
\|u_{xx}\|^p_{\bH^{\ga}_p(-\infty,\infty)}\leq N(d,p,\delta,K)\;\|u_t-A^{ij}u_{x^ix^j}\|^p_{\bH^{\ga}_p(-\infty,\infty)}.
\end{equation}

\end{corollary}

\begin{corollary}
                 \label{so crucial}
    Let $0<T<\infty$, $f^i\in \bL_p(T)$, and $u\in \bH^1_p(T)$   satisfies
    $$
u_t=A^{ij}(t)u_{x^ix^j}+f^i_{x^i}, \quad t\in (0,T), x\in \bR^d
$$
 with zero initial condition $u(0)=0$. Then
\begin{equation}
                            \label{eqn so crucial}
\|u_{x}\|^p_{\bL_p(T)}\leq N(d,p,\delta,K)\|f^i\|^p_{\bL_p(T)}.
\end{equation}
$$
\|u\|^p_{\bH^1_p(T)}\leq N(d,p,\delta,K,T)\|f^i\|^p_{\bL_p(T)}.
$$
\end{corollary}
\begin{proof}
Remember
$$
\|f^i_{x}\|_{H^{-1}_p}\leq N\|f^i\|_{L_p},\quad \|u_x\|_{L_p}\leq N(\|u_{xx}\|_{H^{-1}_p}+\|u\|_{L_p(T)}).
$$
By (\ref{eqn single}) with $\gamma=-1$,
\begin{equation}
                           \label{eqn 3.30.1}
\|u_x\|_{\bL(T)}\leq N(\|f^i\|_{\bL(T)}+\|u\|_{L_p(T)}).
\end{equation}
Notice that, for any constant $c>0$, the function $u^c(t,x):=u(c^2t,cx)$ satisfies
$$
u^c_t=A^{ij}(c^2t)u^c_{x^ix^j}+ (cf^i(c^2t,cx))_{x^i}.
$$
Thus for this function (\ref{eqn 3.30.1}) with $c^{-2}T$ in place of $T$ becomes
$$
\|u_x\|_{\bL(T)}\leq N(\|f^i\|_{\bL(T)}+c^{-1}\|u\|_{L_p}).
$$
Now we get (\ref{eqn so crucial}) by taking $c\to \infty$.
\end{proof}

\begin{corollary}
Let $u\in C^{\infty}_0(\bR^{d};\bR^{d_1})$ and $A^{ij}$ be independent of $t$. Then
\begin{equation}
                            \label{eqn single-3111}
\|u_{xx}\|^p_{H^{\ga}_p}\leq N(d,p,\delta,K)\;\|A^{ij}u_{x^ix^j}\|^p_{H^{\ga}_p}.
\end{equation}
\end{corollary}

\begin{proof}
Take a nonnegative smooth function $\eta(t)\in C^{\infty}_0(-1,1)$ so that $\int_{\bR}\eta^p(t)dt=1$. For each $n=1,2,\cdots$, define $\eta_n(t)=n^{-1/p}\eta(t/n)$. Then applying (\ref{eqn single-3}) for $v_n(t,x):=\eta_n(t)u(x)$,
$$
\|u_{xx}\|^p_{H^{\ga}_p}\leq N \|A^{ij}u_{x^ix^j}\|^p_{H^{\ga}_p}+ N\|u\|^p_{H^{\ga}_p}\int_{\bR} |\eta'_n|^p dt
$$
Now it is enough to let $n\to \infty$. The corollary is proved.
\end{proof}

Remember that for any  $t\in \bR$, $(x^1,x')\in \bR^d$, we defined
$$
B_r(x)=(x^1-r,x^1+r)\times B'_r(x'), \quad Q_r(t,x)=(t,t+r^2)\times B_r(x),
$$
where $B'_r(x')$ is the open ball in $\bR^{d-1}$ of radius $r$ with center $x'$. By $C^{\infty}_{loc}(\bR^{d+1};\mathbb{R}^{d_1})$ we denote the set of
$\mathbb{R}^{d_1}$-valued functions $u$ defined on $\bR^{d+1}$ and such
that $\zeta u\in C^{\infty}_0(\bR^{d+1};\mathbb{R}^{d_1})$ for any
$\zeta\in C^{\infty}_0(\bR^{d+1};\mathbb{R})$.

\begin{theorem}
                          \label{thm 5.5.1}
 Let $q\in (1,\infty)$ and $(t,x)\in \bR^{d+1}$. Then there exists a constant $N$, depending only on $q,d,d_1,\delta$ and $K$ so that for any $\lambda\geq 4, r>0$ and $u\in C^{\infty}_{loc}(\bR^{d+1};\bR^{d_1})$, we have
 \begin{eqnarray*}
&&\aint_{Q_r(t,x)}\aint_{Q_r(t,x)}|u_{xx}(s,y)-u_{xx}(r,z)|^q \,dsdydrdz\\
&\leq& N\lambda^{-q}\aint_{Q_{\lambda r}(t,x)}|u_{xx}|^q \,dsdy +N\lambda^{d+2}\aint_{Q_{\lambda r}(t,x)}|u_t+A^{ij}u_{x^ix^j}|^q \,dsdy.
\end{eqnarray*}
  \end{theorem}
  \begin{proof}
  See Theorem 6.1.2 of \cite{kr08}. Actually this theorem is proved when $d_1=1$, and the proof is based on Theorem \ref{thm entire}.
  Since Theorem \ref{thm entire} holds for any $d_1=1,2,\cdots$, the theorem can be proved  by repeating the proof of Theorem 6.1.2 of \cite{kr08} word for word.
  \end{proof}

\begin{corollary}
                      \label{cor 5.17.1}
    Let $u=u(x)\in C^{\infty}_{loc}(\bR^d; \bR^{d_1})$ and $A^{ij}$ be independent of $t$. Then for any $x\in \bR^d$, $\lambda\geq 4$ and $r>0$,
    \begin{eqnarray*}
&&\aint_{B_r(x)}\aint_{B_r(x)}|u_{xx}(y)-u_{xx}(z)|^q dydz\\
&\leq& N\lambda^{-q}\aint_{B_{\lambda r}(x)}|u_{xx}|^q dy +N\lambda^{d+2}\aint_{B_{\lambda r}(x)}|A^{ij}u_{x^ix^j}|^q dy.
\end{eqnarray*}
  \end{corollary}
\vspace{0.2cm}

From now on we consider systems defined on a {\bf{half space}}.
Remember
$$
\bH^{\ga}_{p,\theta}(S,T):=L_p((S,T),H^{\ga}_{p,\theta}), \quad \quad \|u\|^p_{\bH^{\ga}_{p,\theta}(S,T)}:=\int^T_S\|u(t,\cdot)\|^p_{H^{\ga}_{p,\theta}}\,dt.
$$

\begin{lemma}
                    \label{lemma 03.21.2}
  Let $\gamma,\theta\in \bR$ and $p\in (1,\infty)$.

(i) Let $-\infty\leq S<T\leq \infty$ and suppose  $u(t,x)\in C^{\infty}_0(\bR \times \bR^d_+;\bR^{d_1})$ satisfies
$$
u_t+A^{ij}(t)u_{x^ix^j}=f, \quad  (t,x)\in (S,T)\times \bR^d_+
$$
 and assume $u(T,\cdot)=0$ if $T<\infty$.
\begin{equation}
                       \label{eqn 3.21.5}
\|M^{-1}u\|^p_{\bH^{\gamma+2}_{p,\theta}(S,T)}\leq
N(p,d,\theta,\delta,K)\left(\|M^{-1}u\|^p_{\bH^{\gamma+1}_{p,\theta}(S,T)} + \|Mf\|^p_{\bH^{\gamma}_{p,\theta}(S,T)}\right).
\end{equation}

(ii) If $u(x)\in C^{\infty}_0(\bR^d_+;\bR^{d_1})$ and $A^{ij}$ is independent of $t$, then
\begin{equation}
                       \label{eqn 3.21.5555}
\|M^{-1}u\|^p_{H^{\gamma+2}_{p,\theta}}\leq
N(p,d,\theta,\delta,K)\left(\|M^{-1}u\|^p_{H^{\gamma+1}_{p,\theta}} + \|MA^{ij}u_{x^ix^j}\|^p_{H^{\gamma}_{p,\theta}}\right).
\end{equation}

\end{lemma}
\begin{proof}
(i). We proceed as in  the proof of Lemma 5.8 of \cite{kr99}. Denote $S_n=e^{-2n}S$ and $T_n=e^{-2n}T$. By Lemma \ref{lemma 1}(iii) and  (\ref{def weight}),
\begin{eqnarray}
\|M^{-1}u\|^p_{\bH^{\gamma+2}_{p,\theta}(S,T)}&\leq& N\sum_{n=-\infty}^{\infty}e^{n(\theta-p)}\|\zeta(x)u(t,e^nx)\|^p_{\bH^{\gamma+2}_p(S,T)}\nonumber\\
&=&N\sum_{n=-\infty}^{\infty}e^{n(2+\theta-p)}\|\zeta(x)u(e^{2n}t,e^nx)\|^p_{\bH^{\gamma+2}_p(S_n,T_n)}\nonumber\\
&\leq &N\sum_{n=-\infty}^{\infty}e^{n(2+\theta-p)}\|(\zeta(x)u(e^{2n}t,e^nx))_{xx}\|^p_{\bH^{\gamma}_p(S_n,T_n)}, \label{eqn 4.20.1}
 \end{eqnarray}
 where the last inequality is due to (\ref{eqn 5.1.1}). Denote $v^n(t,x)=\zeta(x)u(e^{2n}t,e^nx)$, then it satisfies
 $$
v^n_t+A^{ij}(e^{2n}t)v^n_{x^ix^j}=e^{2n}\zeta(x)f(e^{2n}t,e^nx)+2e^nA^{1j}(e^{2n}t) \zeta_{x}u_{x^j}(e^{2n}t,e^nx)+A^{11}(e^{2n}t)\zeta_{xx}u(e^{2n}t,e^nx)
$$
for $(t,x)\in (S_n,T_n)\times \bR^d_+$. By (\ref{eqn single}),
\begin{eqnarray*}
\|v^n_{xx}\|^p_{\bH^{\gamma}_p(S_n,T_n)}&\leq& Ne^{2np}\|\zeta(x)f(e^{2n}t,e^nx)\|^p_{\bH^{\gamma}_p(S_n,T_n)}\\
&+& Ne^{np}\|\zeta_{x^i}u_{x^j}(e^{2n}t,e^nx)\|^p_{\bH^{\gamma}_p(S_n,T_n)}+N\|\zeta_{xx}u(e^{2n}t,e^nx) \|^p_{\bH^{\gamma}_p(S_n,T_n)},
\end{eqnarray*}
where $N$ is independent of $n$. Plugging this into (\ref{eqn 4.20.1}) one gets
\begin{eqnarray*}
\|M^{-1}u\|^p_{\bH^{\gamma+2}_{p,\theta}(S,T)}&\leq& N\sum_{n=-\infty}^{\infty}e^{n(\theta+p)}\|\zeta(x)f(t,e^nx)\|^p_{\bH^{\gamma}_p(S,T)}\\
&+&N\sum_{n=-\infty}^{\infty}e^{n\theta}\|\zeta_xu_x(t,e^nx)\|^p_{\bH^{\gamma}_p(S,T)}+N\sum_{n=-\infty}^{\infty}e^{n(\theta-p)}\|\zeta_{xx}u(t,e^nx)\|^p_{\bH^{\gamma}_p(S,T)}.
\end{eqnarray*}
This, (\ref{eqn 5.6.1}) and Lemma \ref{lemma 1} easily lead us to (\ref{eqn 3.21.5}). Indeed, for instance, by (\ref{eqn 5.6.1})
$$
\sum_{n=-\infty}^{\infty}e^{n\theta}\|\zeta_xu_x(t,e^nx)\|^p_{\bH^{\gamma}_p(S,T)}\leq N \|u_x\|^p_{\bH^{\gamma}_{p,\theta}(S,T)}
$$
and by Lemma \ref{lemma 1}(iv) applied to $M^{-1}u$ in place of $u$,
$$
\|u_x\|_{\bH^{\gamma}_{p,\theta}(S,T)}=\|DM(M^{-1}u)\|_{\bH^{\gamma}_{p,\theta}(S,T)}\leq N\|M^{-1}u\|_{\bH^{\gamma+1}_{p,\theta}(S,T)}.
$$
(ii) This is proved similarly based on (\ref{eqn single-3111}). The lemma is proved.
\end{proof}

\begin{remark}
                  \label{remark 3.21.6}
Let $\gamma\geq 0$. By iterating (\ref{eqn 3.21.5}), one gets
\begin{eqnarray*}
\|M^{-1}u\|^p_{\bH^{\gamma+2}_{p,\theta}(S,T)}&\leq& N\|M^{-1}u\|^p_{\bL_{p,\theta}(S,T)}+N\|Mf\|^p_{\bH^{\gamma}_{p,\theta}(S,T)}\\
&\leq&N\|Mu_{xx}\|^p_{\bL_{p,\theta}(S,T)}+N\|Mf\|^p_{\bH^{\gamma}_{p,\theta}(S,T)},
\end{eqnarray*}
where for the second inequality we use (\ref{eqn 4.25.3})  twice.
We use both inequalities later to  estimate $\|M^{-1}u\|^p_{\bH^{\gamma+2}_{p,\theta}(S,T)}$.
\end{remark}

\vspace{3mm}

Let $(w^1_t,w^2_t,\cdots,w^d_t)$ be a $d$-dimensional Wiener process defined on a probability space $(\Omega',\cF,P)$. Denote
$$
\xi_t=w^1_t\sqrt{2}+2t,\quad \eta_t=(\sqrt{2}\int^t_0 e^{\xi_s}dw^2_s,\cdots, \sqrt{2}\int^t_0 e^{\xi_s}dw^d_s)
$$
and define $d\times d$ matrix-valued process $\sigma_t$ so that $(\sigma_tx)^1=e^{\xi_t}x^1$ and $(\sigma_tx)'=x'+x^1\eta_t$. It is easy to check (see \cite{kr99}, p.1628) that $x_t(x):=\sigma_tx$ is the unique solution of the stochastic differential equation
$$
dx_t=\sqrt{2}x^1_t dw_t+ 3x^1_t e_1 dt, \quad x_0(x)=x,
$$
where $e_1=(1,0,\cdots,0)$.
For any $f\in C^{\infty}_0(\bR^d_+)$ and $x\in \bR^d$, define
$$
\mathcal{E}f(x)=\mathbb{E}\int^{\infty}_0 f(\sigma_t x)\,dt:=\int_{\Omega'} \int^{\infty}_0 f(\sigma_t x)\,dtdP.
$$
(See below for the convergence of this integral). Note that if $x^1\leq 0$ then $(\sigma_t x)^1\leq 0$ and thus $\mathcal{E}f(x)=0$.  Denote
$$\cL u:=M^2\Delta u+3MD_1u=\sum_{i=1}^d (MD_i)^2+2MD_1.
$$
\begin{lemma}
                     \label{please}
Let $f\in C^{\infty}_0(\bR^d_+)$.

(i)  $\mathcal{E}f\in L_p(\bR^d)$ and $f=\cL (\mathcal{E}f)$ in the sense of distributions on $\bR^d$.

(ii) There exist $f^1,f^2,\cdots,f^d\in L_p(\bR^d)$ so that $f=MD_if^i$ in the sense of distributions on $\bR^d$, and
$$
\sum_{i=1}^d \|f^i\|_{L_p(\bR^d)}\leq N \|f\|_{L_p(\bR^d_+)}.
$$
\end{lemma}

\begin{proof}
By Theorem 2.11  of \cite{kr99}  (with $\theta=d$ and $b=3$ there), the map $\cL$ is   a bounded one-to-one operator from $H^{2}_{p,d}$ onto  $L_{p,d}$, and its inverse ($:=\cL^{-1}$) is also bounded. Denote $u:=\cL^{-1}f\in H^{2}_{p,d}$.
By Lemma \ref{lemma 1}(i), there exists a sequence $u_n\in C^{\infty}_0(\bR^d_+)$ so that $u_n \to u$ in $H^2_{p,d}$. Denote $f_n(x):=\cL u_n(x)$ for each $x\in \bR^d$. Then
\begin{equation}
                                             \label{eqn 5.19.2}
\cL u_n \to \cL u \,\,(=f) \quad \text{in}\quad L_{p,d}\quad \quad \text{and}\quad \quad \|u_n-u_m\|_{H^2_{p,d}}\leq N   \|f_n-f_m\|_{L_{p,d}}.
 \end{equation}
Obviously $u_n(x)=f_n(\sigma_tx)=0$ if $x^1\leq 0$. By   It\^o's formula (see (2.10) in \cite{kr99} for details), we get
$$
u_n(x)=\mathbb{E}\int^{\infty}_0 f_n(\sigma_t x)\,dt, \quad \quad \forall x\in \bR^d.
$$
The convergence of this improper integral is discussed in the proof of Theorem 2.11 of \cite{kr99}. Actually there it is shown that for any $h\in C^{\infty}_0(\bR^d_+)$ (here, $\theta=d$ and $b=3$ in our case),
\begin{equation}
                              \label{eqn 5.19.1}
      \mathbb{E}\int^{\infty}_0 \|h(\sigma_t x)\|_{L_{p,d}}dt \leq N  \|h\|_{L_{p,d}}\int^{\infty}_0 e^{-(\theta-d+1)(b-1)t+(\theta-d+1)^2t} dt = N  \|h\|_{L_{p,d}},
 \end{equation}
which also implies
$$
 \|u_n-\mathcal{E}f\|_{L_{p,d}}=\|\mathbb{E}\int^{\infty}_0 f_n(\sigma_t x)\,dt-\mathbb{E}\int^{\infty}_0 f(\sigma_t x)\,dt\|_{L_{p,d}}\leq N  \|f_n-f\|_{L_{p,d}} \to 0 \quad \text{as}\quad n\to \infty.
$$
Note $L_{p,d}=L_p(\bR^d_+)$. Since $u_n(x),f_n(x),f(x)$ and $\mathcal{E}f$ vanish if $x^1\leq 0$,  it follows that
\begin{equation}
                                    \label{eqn 5.19.4}
\|u_n-\mathcal{E}f\|_{L_p(\bR^d)}\to 0, \quad \quad \|f_n-f\|_{L_p(\bR^d)}\to 0
\end{equation}
as $n\to \infty$. Also (\ref{eqn 5.19.2}) and fact $\|u_n\|_{H^2_{p,d}}=\|\cL^{-1}f_n\|_{H^2_{p,d}}\leq N\|f_n\|_{L_{p,d}}$ show that $\{MD u_n :n=1,2,\cdots\}$ is a Cauchy sequence in $L_p(\bR^d)$. Indeed, since  each $u_n$ has compact support in $\bR^d_+$,
$$
\|MDu_n-MDu_m\|_{L_p(\bR^d)}=\|MDu_n-MDu_m\|_{L_{p,d}}\leq N\|u_n-u_m\|_{H^1_{p,d}}\leq N\|f_n-f_m\|_{L_{p,d}}.
$$
 Let $\cL^*$ denote the adjoint operator of $\cL$. For any $\phi \in C^{\infty}_0(\bR^d)$, by (\ref{eqn 5.19.4}),
$$
(f,\phi)=\lim_{n\to \infty}(f_n,\phi)=\lim_{n\to \infty}(\cL u_n, \phi)=\lim_{n\to \infty}(u_n, \cL^* \phi)=(\mathcal{E}f,\cL^* \phi)=(\cL (\mathcal{E}f),\phi).
$$
Thus  $f=\cL (\mathcal{E}f)$ in the sense of distributions on $\bR^d$. Also since $u_n \to \mathcal{E}f$ in $L_p(\bR^d)$ and $\{MDu_n\}$ is a Cauchy sequence in $L_p(\bR^d)$, we have $MD  \mathcal{E}f\in L_p(\bR^d)$. Consequently,
$$
f=\cL (\mathcal{E}f)=MD_1(MD_1 \mathcal{E}f+ 2 \mathcal{E}f)+ \sum_{j=2}^d MD_j \mathcal{E}f=:\sum_{i=1}^d MD_if^i,
$$
and by (\ref{eqn 5.19.4}),
$$\sum_i \|f^i\|_{L_p(\bR^d)}=\lim_{n\to \infty}(\|u_n\|_{L_p}+\|MDu_n\|_{L_p})\leq \lim_{n\to \infty} \|u_n\|_{H^2_{p,d}}\leq N\|f_n\|_{L_{p,d}} =N\|f\|_{L_{p,d}}.
$$
The lemma is proved.
\end{proof}

Now we prove a version of Theorem \ref{main theorem} for $\theta=d$.

\begin{lemma}
                                         \label{lemm 12.08.7}
Let $-\infty<S<T <\infty$, $p\in (1,\infty)$ and $n=0,1,2\cdots $.  For any $f\in M^{-1}\bH^n_{p,d}(S,T)$, the equation
$$
u_t+A^{ij}(t)u_{x^ix^j}=f, \quad  (t,x)\in (S,T)\times \bR^d_+
$$
with the condition $u(T)=0$ has a unique solution $u\in \frH^{n+2}_{p,d}(S,T)$, and for this solution
\begin{equation}
                             \label{eqn 3.31.1}
\|M^{-1}u\|_{\bH^{n+2}_{p,d}(S,T)}\leq
N(p,d,\delta,K)\|Mf\|_{\bH^n_{p,d}(S,T)}.
\end{equation}
\end{lemma}
\begin{proof}
As usual we only need to prove that the estimate (\ref{eqn 3.31.1}) holds given that a solution $u$ already exists. Furthermore we may assume $u(t,x)\in C^{\infty}_0(\bR \times \bR^d_+;\bR^{d_1})$.
 Due to Remark \ref{remark 3.21.6} and the inequality $\|M^{-1}u\|_{L_{p,d}}\leq N(p,d)\|u_x\|_{L_{p,d}}$ (see Lemma \ref{lemma 1}(iv)), we
 only need to prove
 \begin{equation}
               \label{eqn 3.1}
 \|u_x\|_{\bL_{p,d}(S,T)}\leq N\|Mf\|_{\bL_{p,d}(S,T)}.
 \end{equation}
 By Lemma \ref{please}, we can write
$Mf=MD_if^i$ on $\bR^d$ (thus $f=D_if^i$), where
 $f^i=(f^{i1},\cdots,f^{id_1})$, so that $f^i \in \bL_{p}(S,T)$ (not only in $\bL_{p,d}(S,T)$) and
$$
\sum_{i=1}^d \|f^i\|_{\bL_{p}(S,T)}\leq N \|Mf\|_{\bL_{p,d}(S,T)}.
$$
Thus by Corollary \ref{so crucial},
$$
\|u_x\|_{\bL_{p,d}(S,T)}=\|u_x\|_{\bL_p(S,T)}\leq N \|f^i\|_{\bL_p(S,T)}\leq N\|Mf\|_{\bL_{p,d}(S,T)}.
$$
The lemma is proved.
\end{proof}

For $r,a>0$, denote
$$
Q_r(a)=Q_r(0,a,0)=(0,r^2)\times (a-r,a+r)\times B'_r(0),\quad U_r=(-r^2,r^2)\times (-2r,2r)\times B'_r(0).
$$

\begin{lemma}
                 \label{lemma 3}
Let $0<s<r <\infty$, $u(t,x)\in C^{\infty}_0(\bR\times \bR^d_+;\bR^{d_1})$ and
$$
u_t+A^{ij}(t) u_{x^ix^j}=0  \quad \quad \text{for} \quad (t,x)\in Q_r(r).
$$
 Then for any multi-index  $\beta=(\beta^1,\cdots,\beta^d)$
  there exists a constant $N= N(p,|\beta|)$ so  that the inequality
\begin{eqnarray}
                                   \label{eqn new}
&& \int_{Q_s(s)}\left(|M^{-1}D^{\beta}u|^p+|D^{\beta}u_x|^p+|MD^{\beta}u_{xx}|^p\right)(x^1)^{\theta-d}dxdt \nonumber\\
 &\leq& N(1+r)^{|\beta|p}\cdot (1+(r-s)^{-2})^{(|\beta|+1)p}
\int_{Q_r(r)}|Mu(t,x)|^p (x^1)^{\theta-d} dxdt
\end{eqnarray}
  holds for $\theta=d$.
\end{lemma}

\begin{proof}
To prove (\ref{eqn new}) we use induction on $|\beta|$. Firstly, consider the case $|\beta|=0$. We modify the proof of Lemma 2.4.4 of \cite{kr08}.  Denote $r_0=s$ and     $r_m=s+(r-s)\sum_{j=1}^m2^{-j}$ for $m=1,2,\cdots$.
choose a smooth function $\zeta_m$ so that  $0\leq \zeta_m\leq 1$,
$$
\zeta_m=1 \quad \text{on} \quad U_{r_m}, \quad \quad \zeta_m=0 \quad \text{on}\quad \Omega \setminus U_{r_{m+1}},
$$
$$
|\zeta_{mx}|\leq N (r-s)^{-1}2^m, \quad |\zeta_{mxx}|\leq N(r-s)^{-2} 2^{2m},\quad |\zeta_{mt}|\leq N (r-s)^{-2} 2^{2m}.
$$
Note that $(u\zeta_m)(r^2,x)=0$ on $\bR^d_+$, and it satisfies
$$
(u\zeta_m)_t+A^{ij}(u\zeta_m)_{x^ix^j}=\zeta_{mt}u+2A^{ij}(u\zeta_{m+1})_{x^i}\zeta_{mx^j}+A^{ij}u\zeta_{mx^ix^j}=:f_m, \quad \quad (t,x)\in (0,r^2)\times \bR^d_+.
$$
By Lemma \ref{lemm 12.08.7} for $\gamma=0$,
$$
A_m:=\|M^{-1}u\zeta_m\|_{\bH^2_{p,d}(r^2)}\leq N \|Mf_m\|_{\bL_{p,d}(r^2)}.
$$
Denote $B:= \left(\int_{Q_r(r)}|Mu|^p dxdt\right)^{1/p}$. Then
$$
\|\zeta_{mt}Mu+A^{ij}Mu\zeta_{mx^ix^j}\|_{\bL_{p,d}(r^2)}\leq N(r-s)^{-2}2^{2m}(\int_{Q_r(r)}|Mu|^p dxdt)^{1/p}
=N(r-s)^{-2}2^{2m}B,
$$
$$
\|A\zeta_{mx}M(u\zeta_{m+1})_{x}\|_{\L_{p,d}(r^2)}\leq N(r-s)^{-1}2^{m}\|M(u\zeta_{m+1})_x\|_{\bL_{p,d}(r^2)}
\leq N(r-s)^{-1}2^{m}\|u\zeta_{m+1}\|_{\bH^1_{p,d}(r^2)},
$$
and  by Lemma \ref{lemma 1} (v) (take $p_0=p_1=p, \gamma=1,\gamma_0=0,\gamma_1=2,\theta=d,\theta_0=d+p,\theta_1=d-p$ and $\kappa=1/2$)  for any $\varepsilon>0$
$$
(r-s)^{-1}2^{m}\|u\zeta_{m+1}\|_{\bH^1_{p,d}(r^2)}\leq \varepsilon A_{m+1}+ \varepsilon^{-1}(r-s)^{-2}2^{2m}B.
$$
It follows (with $\varepsilon$ different from the one above),
$$
A_{m}\leq \varepsilon A_{m+1}+ N(1+\varepsilon^{-1})(r-s)^{-2}2^{2m}B.
$$
We take $\varepsilon=\frac{1}{16}$ and get
$$
\varepsilon^m A_m\leq  \varepsilon^{m+1}A_{m+1}+N\varepsilon^m(1+\varepsilon^{-1})2^{2m}(r-s)^{-2}B,
$$
$$
A_0+\sum_{m=1}^{\infty}\varepsilon^mA_m\leq \sum_{m=1}^{\infty}\varepsilon^mA_m+N(r-s)^{-2}B.
$$
Note that the series $\sum_{m=1}\varepsilon^mA_m$ converges because  $A_m\leq N2^{2m}\|M^{-1}u\|_{\bH^2_{p,d}(r^2)}$.   By Lemma \ref{lemma 1}(iii), for any $M^{-1}w\in H^2_{p,\theta}$,
\begin{equation}
                              \label{eqn 4.24.6}
      \|M^{-1}w\|_{H^2_{p,\theta}}     \sim (\|M^{-1}w\|_{L_{p,\theta}}+\|w_x\|_{L_{p,\theta}}+\|Mw_{xx}\|_{L_{p,\theta}}).
 \end{equation}
 Therefore,
$$
\int_{Q_s(s)}\left(|M^{-1}u|^p+|u_x|^p+|Mu_{xx}|^p\right) dxdt\leq N A^p_0\leq N(r-s)^{-2p}\int_{Q_r(r)}|u(t,x)|^p (x^1)^p dxdt.
$$

Next assume that (\ref{eqn new}) holds  whenever $s<r$ and  $|\beta'|=k$, that is
\begin{eqnarray*}
&& \int_{Q_s(s)}\left(|M^{-1}D^{\beta'}u|^p+|D^{\beta'}u_x|^p+|MD^{\beta'}u_{xx}|^p\right)(x^1)^{\theta-d}dxdt \nonumber\\
 &\leq& N(1+r)^{kp}\cdot (1+ (r-s)^{-2})^{(k+1)p}
\int_{Q_r(r)}|Mu(t,x)|^p (x^1)^{\theta-d} dxdt
\end{eqnarray*}
Let $|\beta|=k+1$ and $D^{\beta}=D_iD^{\beta'}$ for some $i$ and $\beta'$ with $|\beta'|=k$.  Fix
  a smooth function $\eta$  so that  $\eta=1$ on $U_s$, $\eta=0$ on $\Omega\setminus U_{(r+s)/2}$, $|\eta_x|\leq N(r-s)^{-1}, |\eta_{xx}|\leq N(r-s)^{-2}$ and
  $|\eta_t|\leq N(r-s)^{-2}$.
 Note that $v:=\eta D^{\beta}u$ satisfies $v(r^2,\cdot)=0$ and
$$
v_t+A^{ij}v_{x^ix^j}=f:=\eta_{t}D^{\beta}u+2A^{ij}\eta_{x^i}D^{\beta}u_{x^j}+A^{ij}\eta_{x^ix^j}D^{\beta}u, \quad \quad (t,x)\in (0,r^2)\times \bR^d_+.
$$
By Lemma \ref{lemm 12.08.7} for $\gamma=0$ (also note that $x^1\leq r$ on the support of $\eta$ and $(r-s)^{-1}\leq 1+(r-s)^{-2}$),
\begin{eqnarray*}
\|M^{-1}v\|^p_{\bH^2_{p,d}(r^2)} &\leq& N \|M\eta_{t}D^{\beta}u+2A\eta_{x}MD^{\beta}u_x+MA\eta_{xx}D^{\beta}u\|^p_{\bL_{p,d}(r^2)}\\
&\leq&N (1+r)^p(1+(r-s)^{-2})^p \int_{Q_{(s+r)/2}((s+r)/2)}\left(|D^{\beta}u|^p+|MD^{\beta}u_x|^p\right)dxdt\\
&\leq&N (1+r)^p(1+(r-s)^{-2})^p \int_{Q_{(s+r)/2}((s+r)/2)}\left(|D^{\beta'}u_x|^p+|MD^{\beta'}u_{xx}|^p\right)dxdt.
\end{eqnarray*}
This and (\ref{eqn 4.24.6}) show that  the induction goes through, and hence the lemma is proved.
\end{proof}

\begin{remark}
                           \label{2011.1.15}
 The proof of Lemma \ref{lemma 3} mainly depends on Lemma \ref{lemm 12.08.7} and it can be easily checked that   the assertion of Lemma \ref{lemma 3} holds for $\theta=\theta_0$
whenever Lemma \ref{lemm 12.08.7} is true for $\theta=\theta_0$. Thus due to Theorem \ref{main theorem} (which will be proved in section \ref{section proof}),  Lemma \ref{lemma 3} holds for $\theta\in (d+1-p, d+p-1)$ if $p\in (1,2]$ and $\theta\in (d-1,d+1)$ if $p\in (2,\infty)$.
\end{remark}

\begin{lemma}
                   \label{lemma 1111}
Let $u(t,x)\in C^{\infty}_0(\bR\times \bR^d_+;\bR^{d_1})$. Then for any
$T>0$, $p>1$ and $n=0,1,2,\cdots$,
$$
\sup_{t\in [0,T]}\|u(t,\cdot)\|_{H^n_{p,\theta}}\leq
N(\|u\|_{\bH^n_{p,\theta}(T)}+\|u_t\|_{\bH^n_{p,\theta}(T)}).
$$
\end{lemma}
\begin{proof}
See p. $66$ of \cite{kr08};  actually in this book, weights are
not used and hence we give an outline of  the proof.
First of all, it is easy to check that for any $\phi=\phi(t)\in W^1_p((0,T))$ (cf. p.32 of \cite{kr08})
$$
\sup_{t\leq T}|\phi(t)|^p\leq N \int^T_0(|\phi|^p+|\phi'(t)|^p)dt.
$$
Thus  it suffices to prove
\begin{equation}
                                          \label{eqn 12.19}
\phi(t):=\|u(t,\cdot)\|_{H^n_{p,\theta}}\in W^1_p((0,T)), \quad |\phi'(t)|\leq\|u_t(t,\cdot)\|_{H^n_{p,\theta}}.
\end{equation}
One can  prove (\ref{eqn 12.19}) by repeating  the proof of Exercise 2.4.8 on p.71  of \cite{kr08}. It is enough to replace $H^n_p$ there by $H^n_{p,\theta}$.
\end{proof}

\begin{lemma}
                                                      \label{101004.14.53}
Let $\theta\leq d$, $p>1$, $s\in (0,r)$ and  $u\in C^{\infty}_{loc}(\Omega;\mathbb{R}^{d_1})$ satisfies\;
$u_t+A^{ij}(t)u_{x^ix^j}=0$ for $(t,x)\in Q_r(r)$. Then for any multi-index $\beta=(\beta^1,\beta^2,\cdots,\beta^d)$,
\begin{eqnarray}
\max_{(t,x)\in Q_s(s)}(|D^{\beta}u_{xx}|^p+|D^{\beta}u_{t}|^p)\le N
\int_{Q_r(r)}|u|^p (x^1)^{\theta-d+p}dxdt,\nonumber
\end{eqnarray}
where  $N=N(s,r,\beta,p,\delta,K)$.
\end{lemma}

\begin{proof}  Choose the smallest integer $n$ so that $np>d$.
Let $v\in C^{\infty}_0(\bR^d_+)$ satisfy $v(x)=0$ for $x^1 \geq 2r$.  The by Lemma \ref{lemma 1} $(ii)$ with $\gamma=n$, $i=0$, $\theta=d$ and $u=M^{-n}v$,
\begin{equation}
              \label{eqn 12.10.1}
\sup_{x}|v(x)|\leq N(r)\sup_{x}|M^{d/p}M^{-n}v(x)|\leq N\|M^{-n}v\|_{H^n_{p,d}}\leq N(r,p,n)\|D^nv\|_{L_{p,d}},
\end{equation}
where for the last inequality we use Remark \ref{remark last}.

Fix $\kappa\in (s,r)$. Let  $\psi$  be a smooth function so that
$\psi(x)=1$ for $(t,x)\in Q_{s}(s)$ and $ \psi=0$ for $(t,x)\not\in U_{\kappa}$.  It follows from (\ref{eqn 12.10.1}) and Lemma \ref{lemma 1111} that
\begin{eqnarray*}
\max_{ Q_{s}(s)} \left(|D^{\beta}u_{xx}|+|D^{\beta}u_{t}|\right)&\leq& N\max_{(t,x)\in Q_{s}(s)} | (D^{\beta}\psi u)_{xx}|\\
&\leq &N\max_{t\in [0,s^2]}\|D^n
(D^{\beta}\psi u)_{xx}\|_{L_{p,d}}
\\
 &\leq& N
\left(\|D^n(D^{\beta}\psi u)_{xx}\|_{\bL_{p,d}(s^2)}+\| D^n(D^{\beta}\psi u_t)_{xx}\|_{\bL_{p,d}(s^2)}\right)\\
&\leq &N \sum_{|\alpha|\leq n+|\beta|+4}\int_{Q_{\kappa}(\kappa)}|D^{\alpha}u|^p \,dxdt\\
& \leq& N  \int_{Q_r(r)}|u|^p(x^1)^pdxdt\leq N \int_{Q_r(r)}|u|^p(x^1)^{\theta-d+p}dxdt,
\end{eqnarray*}
where the last inequality is due to the fact that  $1\leq N(r)(x^1)^{\theta-d}$ for $x^1\leq 2r$. The lemma is proved.
 \end{proof}

\begin{remark}
                              \label{remark 55555}
Actually by inspecting  the  proof of Lemma \ref{101004.14.53} it can be easily shown  that if Lemma \ref{lemma 3} holds for some $\theta_0\in (d-1,d-1+p)$  then  Lemma \ref{101004.14.53} holds for any $\theta \in (d-1,\theta_0]$.
\end{remark}

\mysection{Main estimates : Sharp function estimations}
       \label{section sharp}

Remember that we denote
$$\nu_{\al}(dx)=\nu^1_{\al}(dx^1)dx':=(x^1)^{\alpha}dx^1dx'.
$$

The following is a weighted version of Poincar\'e's inequality.
\begin{lemma}\label{2010.03.23.3}
Let $\alpha \geq 0$, $p\in [1,\infty)$, $D_r(a):=(a-r,a+r)\times B'_r(0)\subset \mathbb{R}^d_+$ ,
and $u\in C^{\infty}_{loc}(\mathbb{R}^d_+;\mathbb{R}^{d_1})$. Then
\begin{eqnarray}
\int_{D_r(a)}\int_{D_r(a)}|u(x)-u(y)|^p \nu_{\al}(dx)\; \nu_{\al}(dy)\le
2^{\alpha+2} (2r)^p|D_r(a)|\int_{D_r(a)}|u_x(x)|^p
\nu_{\al}(dx),\label{2010.03.17.1}
\end{eqnarray}
where $|D_r(a)|:=\nu_{\al}(D_r(a))$ and we
define
\begin{eqnarray}
\int_A |f(x)|^p\nu_{\al}(dx)=\sum_{k=1}^{d_1} \int_A
|f^k(x)|^p\nu_{\al}(dx)\nonumber
\end{eqnarray}
for $\mathbb{R}^{d_1}$-valued function $f$ and $A\subset \Omega$.
\end{lemma}
\begin{proof}
We use the outline of the proof of Theorem 10.2.5 of \cite{kr08}.  Without loss of generality we may assume $d_1=1$.
For $x,y\in D_r(a)$  we have
\begin{eqnarray}
|u(x)-u(y)|^p\le (2r)^p\int^1_0 |u_x(tx+(1-t)y)|^pdt\nonumber
\end{eqnarray}
and the left-hand side of $(\ref{2010.03.17.1})$ is less than
\begin{eqnarray}
(2r)^p\int^1_0 I(t)dt=2(2r)^p \int^1_{1/2} I(t)dt,\nonumber
\end{eqnarray}
where
\begin{eqnarray}
I(t):=\int_{D_r(a)}\int_{D_r(a)}|u_x(tx+(1-t)y)|^p \nu_{\al}(dx)\;
\nu_{\al}(dy)\nonumber
\end{eqnarray}
and $I$ satisfies $I(t)=I(1-t)$. For each $t\in [1/2,1]$ and $y$, substituting $w=tx+(1-t)y$ and noticing $x^1=(w^1-(1-t)y^1)/t\leq w^1/t$ since $y^1\geq 0$, we get
\begin{eqnarray}
I(t)&\leq&t^{-\alpha-1}\int_{D_r(a)}\left(\int_{tD_r(a)+(1-t)y}|u_x(w)|^p\nu_{\al}(dw)\right)\nu_{\al}(dy)\nonumber\\
&\le&
2^{\alpha+1}\int_{D_r(a)}\left(\int_{D_r(a)}|u_x(x)|^p\nu_{\al}(dx)\right)\nu_{\al}(dy)\nonumber\\
&=&2^{\alpha+1}|D_r(a)|\int_{D_r(a)}|u_x(x)|^p\nu_{\al}(dx)\nonumber
\end{eqnarray}
with the observation $tD_r(a)+(1-t)y:=\{tz+(1-t)y:z\in
D_r(a)\}\subset D_r(a)$. Now, (\ref{2010.03.17.1}) follows.
\end{proof}

  \begin{lemma}\label{2010.03.23.2}
Let $\alpha>-1$. Recall $\nu_{\alpha}^1(dx^1)=(x^1)^{\alpha}dx^1$. For any $B^1_r(a)\subset \mathbb{R}_+$ we have a non-negative
function $\zeta\in C^{\infty}_0(\mathbb{R}_+;\mathbb{R})$ such that
\begin{eqnarray}
supp(\zeta)\in B^1_{r/2}(a),\quad \int_{B^1_r(a)}\zeta(x^1)\nu^1_{\alpha}(dx^1)=1,\quad
\sup\zeta \cdot |B^1_r(a)|\le N,\quad \sup |\zeta_{x^1}|\cdot |B^1_r(a)|\le
\frac{N}{r},\label{2010.03.19.1}
\end{eqnarray}
where $N=N(\alpha)$ and $|B^1_r(a)|=\nu_{\alpha}^1(B^1_r(a))$.
\end{lemma}

\begin{proof}
Choose a nonnegative smooth function $\psi=\psi(x^1)\in C^{\infty}_0(B^1_{1/2}(0))$ so that $\int_{\bR} \psi(x^1) dx^1=1$.
Define
$$
\zeta(x^1)=\frac{(x^1)^{-\alpha}}{r}\psi(\frac{x^1-a}{r}).
$$
Then the first and the second of (\ref{2010.03.19.1}) are obvious.

\noindent
{\bf{Case 1}}: Let $\alpha \geq 0$. Since $r\leq a$ and $(a+r)^{\alpha+1}-(a-r)^{\alpha+1}\leq 2r(\alpha+1) (2a)^{\alpha}$, the third follows:
\begin{eqnarray}
\sup |\zeta| \cdot |B^1_r(a)|\leq &N&\sup_{|x^1-a|\leq r/2} \frac{(x^1)^{-\alpha}}{r} \cdot ((a+r)^{\alpha+1}-(a-r)^{\alpha+1})\label{eqn 5.23.5}\\
&\leq& N \frac{(a/2)^{-\alpha}}{r}\cdot ((a+r)^{\alpha+1}-(a-r)^{\alpha+1})\leq N.\nonumber
\end{eqnarray}
Similarly, the last inequality holds by
\begin{eqnarray}
\sup |\zeta_{x^1}| \cdot |B^1_r(a)| &\leq& N\sup_{|x^1-a|\leq r/2} \left(\frac{(x^1)^{-\alpha}}{r^2}+ \frac{(x^1)^{-\alpha-1}}{r}\right)\cdot ((a+r)^{\alpha+1}-(a-r)^{\alpha+1}) \nonumber\\
&\leq& \frac{N}{r}(1 +  \frac{(2a)^{\alpha+1}}{(a/2)^{\alpha+1}})\leq \frac{N}{r}. \nonumber
\end{eqnarray}
{\bf{Case 2}}: Let $\alpha \in (-1,0)$.  First assume $r\leq a/2$. Then  by mean value theorem $(a+r)^{\alpha+1}-(a-r)^{\alpha+1}\leq 2r(\alpha+1) (a/2)^{\alpha}$ and thus the right term of (\ref{eqn 5.23.5}) is bounded by a constant $N$.
If $r \in [a/2,a]$, then
$$
\sup_{|x^1-a|\leq r/2} \frac{(x^1)^{-\alpha}}{r} \cdot ((a+r)^{\alpha+1}-(a-r)^{\alpha+1})\leq \frac{(2a)^{-\alpha}}{a/2}(2a)^{\alpha+1}\leq N.
$$
One can handle \,\,$\sup |\zeta_{x^1}| \cdot |B^1_r(a)|$\,\, similarly.
The lemma is proved.
\end{proof}

Now we consider the system
\begin{eqnarray}
u_t+A^{ij}u_{x^ix^j}=f^i_{x^i}+g,\quad (t,x)\in
\Omega=\mathbb{R}\times \mathbb{R}^d_+;\quad f^i=(f^{1i},\ldots,f^{d_1 i}), \label{backward sys}
\end{eqnarray}
i.e.,
\begin{eqnarray}
u^k_t+a^{ij}_{kr}u^r_{x^ix^j}=f^{ki}_{x^i}+g^k,\quad k=1,2,\ldots,d_1.\nonumber
\end{eqnarray}

Recall that for $t\in \bR$, $a\in\bR_+$ and  $x'\in \bR^{d-1} $
$$
Q_r(t,a,x'):=(t,t+r^2)\times (a-r,a+r)\times B'_r(x'), \quad  Q_r(a):=Q_r(0,a,0).
$$
By $C^{\infty}_{loc}(\Omega;\mathbb{R}^{d_1})$ we denote the set of
$\mathbb{R}^{d_1}$-valued functions $u$ defined on $\Omega$ and such
that $\zeta u\in C^{\infty}_0(\Omega;\mathbb{R}^{d_1})$ for any
$\zeta\in C^{\infty}_0(\Omega;\mathbb{R})$.

\begin{lemma}\label{key lemma 2}
Let $\alpha \geq 0$, $p\in [1,\infty)$,
$f^i,g\in C^{\infty}_{loc}(\Omega;\mathbb{R}^{d_1})$. Assume that $u\in C^{\infty}_{loc}(\Omega;\mathbb{R}^{d_1}) $
satisfies (\ref{backward sys}) on $Q_r(a)\subset \Omega$. Then
\begin{eqnarray}
\int_{Q_r(a)}\left|u(t,x)-u_{Q_r(a)}\right|^p\mu_{\alpha}(dtdx)\le N r^p
\int_{Q_r(a)}(|u_x(t,x)|^p+|f(t,x)|^p+r^p|g(t,x)|^p)\mu_{\alpha}(dtdx),\label{2010.03.23.5}
\end{eqnarray}
where   $N=N(\theta,\alpha,p,d,d_1,K)$.
\end{lemma}

\begin{proof}

We follow the outline
of the proof of Theorem 4.2.1 in \cite{kr08}.
We take the scalar function $\zeta$ corresponding to $B^1_r(a)$ and $\alpha$ from
Lemma \ref{2010.03.23.2} and take a nonnegative function $\phi=\phi(x')\in C^{\infty}_0(B'_1(0))$ with unit integral. Denote $\eta(x')=r^{-d+1}\phi(\frac{x'}{r})$, $D_r(a):=(a-r,a+r)\times B'_r(0)$ as before, and for $t\in (0,r^2)$ set
\begin{eqnarray}
\bar u(t):=\int_{D_r(a)}\zeta(y^1)\eta(y')u(t,y)\nu_{\alpha}(dy).\nonumber
\end{eqnarray}
Then by Jensen's inequality and the weighted version of Poincar\'e's
inequality (Lemma \ref{2010.03.23.3}),
\begin{eqnarray}
&&\int_{D_r(a)}|u(t,x)-\bar u(t)|^p\nu_{\alpha}(dx)\nonumber\\
&=&\int_{D_r(a)}\Big|\int_{D_r(a)}(u(t,x)-u(t,y))\zeta(y^1)\eta(y')\nu_{\alpha}(dy)\Big|^p\nu_{\alpha}(dx)\nonumber\\
&\le&
\int_{D_r(a)}\left(\int_{D_r(a)}|u(t,x)-u(t,y)|^p\zeta(y^1)\eta(y')\nu_{\alpha}(dy)\right)\nu_{\alpha}(dx)\nonumber\\
&\le&
|\sup\;\zeta|\cdot |\sup\,\eta|\,\int_{D_r(a)}\int_{D_r(a)}|u(t,x)-u(t,y)|^p\nu_{\alpha}(dx)\nu_{\alpha}(dy)\nonumber\\
&\le& N r^{-d+1}|\sup\;\zeta| \cdot \nu_{\alpha}(D_r(a))\;r^p \int_{D_r(a)}|u_x(t,x)|^p \nu_{\alpha}(dx)\nonumber\\
&\le& N r^{-d+1}|\sup\;\zeta| \cdot \nu_{\alpha}^1(B^1_r(a))\; r^{d-1}r^p \int_{D_r(a)}|u_x(t,x)|^p \nu_{\alpha}(dx)\nonumber\\
&\leq &N\;r^p\int_{D_r(a)}|u_x(t,x)|^p
\nu_{\alpha}(dx). \label{2010.03.23.4}
\end{eqnarray}

We observe that for any constant vector $c\in \mathbb{R}^{d}$ the
left-hand side of (\ref{2010.03.23.5}) is less than $2\cdot 2^p$
times
\begin{eqnarray}
\int_{Q_r(a)}|u(t,x)-c|^p\mu_{\alpha}(dtdx)\le 2^p \int_{Q_r(a)}|u(t,x)-\bar
u(t)|^p\mu_{\alpha}(dtdx)+2^p\;\nu_{\alpha}(D_r(a))\int^{r^2}_{0}|\bar
u(t)-c|^pdt.\nonumber
\end{eqnarray}
By (\ref{2010.03.23.4}) the first term is less than
(\ref{2010.03.23.5}). To estimate the second term, we take
\begin{eqnarray}
c=\frac{1}{r^2}\int^{r^2}_{0} \bar u(t)dt.\nonumber
\end{eqnarray}
Then by Poincar\'e's inequality without a weight in variable $t$ we
have
\begin{eqnarray}
&&\nu_{\alpha}(D_r(a))\int^{r^2}_{0}|\bar u(t)-c|^pdt\nonumber\\
&\le& N\;\nu_{\alpha}(D_r(a))\;(r^2)^p
\int^{r^2}_{0}\Big|\int_{D_r(a)}\zeta(x^1)\eta(x')u_t(t,x)\nu_{\alpha}(dx)\Big|^pdt. \label{2010.03.24.1}
\end{eqnarray}

Remember $u_t=-A^{ij}(t)u_{x^ix^j}+f^i_{x^i}+g$. We show that
(\ref{2010.03.24.1}) is less than (\ref{2010.03.23.5}). In fact, for
handling the integral with $g$, using Jensen's inequality and taking the supremum out of the integral, we have
\begin{eqnarray}
&&\nu_{\alpha}(D_r(a))\;r^{2p}
\int^{r^2}_{0}\Big|\int_{D_r(a)}\zeta(x^1)\eta(x')g(t,x)\nu_{\alpha}(dx)\Big|^pdt\nonumber\\
&\le&\nu_{\alpha}(D_r(a))\; r^{2p}\;|\sup\;\zeta|\,|\sup \eta|\,
\int^{r^2}_{0}\int_{D_r(a)}|g(t,x)|^p\nu_{\alpha}(dx)dt\nonumber\\
&\le&N \nu_{\alpha}^1(B^1_r(a))r^{d-1}\; r^{2p}\;|\sup\;\zeta|\,r^{-d+1}
\int^{r^2}_{0}\int_{D_r(a)}|g(t,x)|^p\nu_{\alpha}(dx)dt\nonumber\\
&\le& N(\theta,p,d)\;r^{2p}\int_{Q_r(a)}|g(t,x)|^p\mu_{\alpha}(dtdx),\nonumber
\end{eqnarray}
where we used $|\sup \zeta|\;\nu_{\alpha}^1(B^1_r(a))\leq N$ (Lemma \ref{2010.03.23.2}).

Next, we handle
 the integral with $-A^{ij}u_{x^ix^j}$. \underline{Fix  $i,j$}.
 Firstly, assume either $i$ or $j$ is $1$; say $j=1$. We use integration by parts
and observe
\begin{eqnarray}
&&\nu_{\alpha}(D_r(a))\;(r^2)^p
\int^{r^2}_{0}\Big|\int_{D_r(a)}\zeta(x^1)\eta(x')A^{ij}(t)u_{x^ix^j}(t,x)\nu_{\alpha}(dx)\Big|^pdt\nonumber\\
&\le&\nu_{\alpha}(D_r(a))\;r^{2p}\int^{r^2}_{0}\Big|\int_{D_r(a)}\zeta_{x^1}(x^1)\eta(x')A^{ij}(t)u_{x^i}(t,x)
\nu_{\alpha}(dx)\Big|^pdt\nonumber\\
&&\quad+\nu_{\alpha}(D_r(a))\;r^{2p}|\alpha|^p\int^{r^2}_{0}\Big|\int_{D_r(a)}\frac{1}{x}\zeta(x^1)\eta(x')A(t)u_{x^i}(t,x)
\nu_{\alpha}(dx)\Big|^pdt\nonumber\\
&=:& \quad I_1+I_2.\nonumber
\end{eqnarray}
For $I_2$ we use the fact $|A^{ij}u_{x^i}|\le |A^{ij}||u_{x^i}|\le K|u_x|$ and $1/x\leq 2/r$ on the support of $\zeta$.  The argument handling
the case of $g$ easily shows
\begin{eqnarray}
I_2\le N(K,\theta,p,d)\;r^{p}\int_{Q_r(a)}|u_x(t,x)|^p\mu_{\alpha}(dtdx).\nonumber
\end{eqnarray}
For $I_1$ we use H\"older's inequality and get
\begin{eqnarray*}
\nu_{\alpha}(D_r(a))\cdot|\int_{D_r(a)} \zeta_{x^1}\eta A^{ij}u_{x^i} \;\nu_{\alpha}(dx)|^p&\leq& \nu_{\alpha}(D_r(a))^{p}\int_{D_r(a)}|\zeta_{x^1} \eta A^{ij}u_{x^i}|^p \;\nu_{\alpha}(dx)\\
&\leq& N(\nu_{\alpha}^1(B^1_r(a))^p r^{(d-1)p}\cdot|\sup \zeta_{x^1}|^p r^{(-d+1)p}\int_{D_r(a)}|u_x|^p\nu_{\alpha}(dx).
\end{eqnarray*}
Since $\nu_{\alpha}^1(B^1_r(a))\cdot |\sup \zeta_x|\leq N/r$, it easily follows that
\begin{eqnarray}
I_1\le N(K,\theta,p,d)\;r^{p}\int_{Q_r(a)}|u_x(t,x)|^p\mu_{\alpha}(dtdx).\nonumber
\end{eqnarray}
Secondly, if $i,j\neq 1$,  by integration by parts, H\"older's inequality and the inequality $\sup|\eta_{x'}|\leq N r^{-d}$,
\begin{eqnarray*}
&&\nu_{\alpha}(D_r(a))\;r^{2p}\int^{r^2}_{0}\Big|\int_{D_r(a)}\zeta(x^1)\eta(x')\left[-A^{ij}(t)u_{x^ix^j}(t,x)\right]\nu_{\alpha}(dx)\Big|^pdt\nonumber\\
&=&
\nu_{\alpha}(D_r(a))\;r^{2p}\int^{r^2}_{0}\Big|\int_{D_r(a)}\zeta(x^1)\eta_{x^j}(x')A^{ij}(t)u_{x^i}(t,x) \nu_{\alpha}(dx)\Big|^pdt\nonumber\\
&\leq&
\nu_{\alpha}(D_r(a))^p\;r^{2p}\int^{r^2}_{0}\int_{D_r(a)}\Big|\zeta(x^1)\eta_{x^j}(x')A^{ij}(t)u_{x^i}(t,x)\Big|^p \nu_{\alpha}(dx)\;dt\nonumber\\
&\leq&
N\nu_{\alpha}(D_r(a))^p\;r^{2p}\cdot\sup |\zeta|^p\cdot r^{-dp}\int^{r^2}_{0}\int_{D_r(a)}|u_{x}|^p \nu_{\alpha}(dx)dt\nonumber\\
&\leq&
N r^{p}\int_{Q_r(a)}|u_{x}|^p \mu_{\alpha}(dxdt).
\end{eqnarray*}

For the integral with $f^i_{x^i}$ we use a similar calculation to the one
of $-A^{ij}u_{x^ix^j}$ and get for each $i$
\begin{eqnarray}
&&\nu_{\alpha}(D_r(a))\,r^{2p}\int^{r^2}_{0}\Big|\int_{D_r(a)}\zeta(x^1) \eta(x')f_{x^i}(t,x)\nu_{\alpha}(dx)\Big|^pdt\nonumber\\
&\le& N(K,\theta,p,d)\;r^{p}\int_{Q_r(a)}|f(t,x)|^p\mu_{\alpha}(dtdx).\nonumber
\end{eqnarray}
The lemma is proved.
\end{proof}


\begin{lemma}\label{key lemma 3}
Let $\alpha \geq 0, p\in [1,\infty)$,  $0<r\leq a$ and $u\in
C^{\infty}_{loc}(\Omega;\mathbb{R}^{d_1})$.

(i)  There is a constant $N=N(K,\theta,\alpha,p,d,d_1)$
such that for any $i=1,\cdots,d$  we have
\begin{eqnarray}
\int_{Q_r(a)}\left|u_{x^i}(t,x)-(u_{x^i})_{Q_r(a)}\right|^p\mu_{\alpha}(dtdx)\le N r^p
\int_{Q_r(a)}(|u_{xx}(t,x)|^p+|u_t(t,x)|^p)\mu_{\alpha}(dtdx) \label{101004.14.17}
\end{eqnarray}

(ii) Denote $\kappa_0=\kappa_0(r,a):=(\nu^1_{\alpha}(B^1_r(a))^{-1}\cdot \int^{a+r}_{a-r} x^1 \nu^1_{\alpha}(dx^1)$. Then
\begin{eqnarray}
&&\int_{Q_r(a)}\left|u(t,x)-u_{Q_r(a)}+\kappa_0(u_{x^1})_{Q_r(a)}-\sum_{i=1}^d x^i(u_{x^i})_{Q_r(a)}\right|^p\mu_{\alpha}(dtdx)\nonumber\\
&\le& N r^p\int_{Q_r(a)}
(|u_{x}(t,x)-(u_x)_{Q_r(a)}|^p+r^p|u_t(t,x)|^p+r^p|u_{xx}(t,x)|^p)\mu_{\alpha}(dtdx)\nonumber\\
&\le& N
r^{2p}\int_{Q_r(a)}(|u_{xx}(t,x)|^p+|u_t(t,x)|^p)\mu_{\alpha}(dtdx)\label{101004.14.22}
\end{eqnarray}
\end{lemma}

\begin{proof}
(i) For (\ref{101004.14.17}) we use the fact that for $v=u_{x^i}$, $v_t-A^{jm}v_{x^jx^m}=(u_t-A^{jm}u_{x^jx^m})_{x^i}$ and apply Lemma \ref{key lemma 2} with $f^i=u_t-A^{jm}u_{x^jx^m}$ for all $i$.

(ii) To prove (\ref{101004.14.22}), denote
$v(t,x):=u(t,x)-(u)_{Q_r(a)}+\kappa_0(u_{x^1})_{Q_r(a)}-\sum_i x^i(u_{x^i})_{Q_r(a)}$.  Then
$$
v_{Q_r(a)}=\kappa_0(u_{x^1})_{Q_r(a)}-\sum_i\frac{(u_{x^i})_{Q_r(a)}}{|Q_r(a)|}\int_{Q_r(a)} x^i \nu_{\alpha}(dx)dt=0,
$$
$$
 v-v_{Q_r(a)}=v, \quad v_{x^i}=u_{x^i}-(u_{x^i})_{Q_r(a)}, \quad v_t-A^{ij}v_{x^ix^j}=g:=u_t-A^{ij}u_{x^ix^j}.
 $$
  Now it is enough to use Lemma \ref{key lemma 2} and (\ref{101004.14.17}). The lemma is proved.
\end{proof}


{\bf{From this point on    we fix \underline{$\alpha:=\theta-d+p$} }} (note $\alpha>0$) and denote
$$
\nu:=\nu_{\alpha},\quad \nu^1:=\nu_{\alpha}^1, \quad   \quad \mu(dxdt)=\nu(dx)dt=(x^1)^{\theta-d+p}dxdt.
$$

\begin{theorem}\label{101004.16.53}
Let $\theta \in (d-1,d]$, $0<r\leq a$ and $\lambda r/a\geq 2$.

(i) Assume that $u\in
C^{\infty}_{loc}(\Omega;\mathbb{R}^{d_1})$ satisfies $u_t+A^{ij}(t)u_{x^ix^j}=0$ in $Q_{\lambda
r}(t_0,a,x'_0)\cap\Omega$. Then there is a constant $N=N(K,\delta,\theta,p,d,d_1)$
so that
\begin{eqnarray}
&&\aint_{Q_r(t_0,a,x'_0)}|u_{xx}(t,x)-(u_{xx})_{Q_r(t_0,a,x'_0)}|^p \mu(dtdx)\nonumber\\
 &\leq&
\frac{N}{(1+\lambda r/a)^p}\;\aint_{Q_{\lambda r}(t_0,a,x'_0)\cap
\Omega}|u_{xx}(t,x)|^p\mu(dtdx).\label{eqn 2.21.2011}
\end{eqnarray}

(ii) If $u\in C^{\infty}_{loc}(\bR^d_+;\bR^{d_1})$, $A^{ij}$ is independent of $t$ and $A^{ij}u_{x^ix^j}=0$ in $B_{\lambda r}(a,x'_0)\cap \bR^d_+$, then
\begin{eqnarray}
&&\aint_{B_r(a,x'_0)}|u_{xx}(x)-(u_{xx})_{B_r(a,x'_0)}|^p \nu(dx)\nonumber\\
 &\leq&
\frac{N}{(1+\lambda r/a)^p}\;\aint_{B_{\lambda r}(a,x'_0)\cap
\bR^d_+}|u_{xx}(x)|^p\nu(dx). \label{eqn 5.13.2011}
\end{eqnarray}

\end{theorem}

\begin{proof}
(ii) is a consequence of (i).  To prove (i), without loss of generality we may assume $t_0=0$, $x'_0=0$ and thus $Q_r(t_0,a,x'_0)=Q_r(a)$.

{\bf{Step 1}}. First, we consider the case $a=1$.  Note that
$$
r\leq 1, \quad 2\leq \lambda r, \quad \beta:=\frac{1+\lambda r}{2} \leq \lambda r, \quad \frac{r}{\beta}\leq \frac{1}{\beta}\leq \frac{2}{3}, \quad \quad 2\beta =1+\lambda r.
$$
Thus,
$$
 Q_{\beta}(\beta) \subset Q_{\lambda r}(1)\cap \Omega, \quad Q_{r/\beta}(\beta^{-1}) \subset Q_{2/3}(2/3).
 $$
 Denote $w(t,x)=u(\beta^2t,\beta x)$, then
 obviously
 $$
w_t+A^{ij}(\beta^2t)w_{x^ix^j}=0, \quad \quad \text{for}\quad (t,x)\in Q_1(1)
$$
and
\begin{eqnarray*}
\aint_{Q_r(1)}|u_{xx}(t,x)-(u_{xx})_{Q_r(1)}|^p (x^1)^{\theta-d+p}dxdt &\leq& N(d)\sup_{Q_r(1)}(|u_{xxx}|^p+|u_{xxt}|^p)\\
&\leq& N(d)\beta^{-3p} \,\,\sup_{Q_{r/\beta}(\beta^{-1})}(|w_{xxx}|^p+|w_{xxt}|^p)\\
&\leq& N(d)\beta^{-3p} \,\,\sup_{Q_{2/3}(2/3)} (|w_{xxx}|^p+|w_{xxt}|^p).
\end{eqnarray*}
Applying  Lemma \ref{101004.14.53} to
$v(t,x)=w(t,x)-w_{Q_1(1)}+\kappa_0(w_{x^1})_{Q_1(1)}-\sum_{i=1}^d x^i(w_{x^i})_{Q_1(1)}$, and then using   Lemma \ref{key lemma 3}
\begin{eqnarray*}
\beta^{-3p} \,\,\sup_{Q_{2/3}(2/3)} (|w_{xxx}|^p+|w_{xxt}|^p) &\leq& N\beta^{-3p}\int_{Q_1(1)}|v|^p (x^1)^{\theta-d+p}dxdt\\
&\leq& N\beta^{-3p}\int_{Q_1(1)}|w_{xx}|^p (x^1)^{\theta-d+p}dxdt\\
&=&N\beta^{-2p-2-\theta}\int_{Q_{\beta}(\beta)}|u_{xx}|^p (x^1)^{\theta-d+p}dxdt.
\end{eqnarray*}
This leads to (\ref{eqn 2.21.2011}) since $|Q_{\lambda r}(1)\cap \Omega|\sim \beta^{p+\theta+2}$.

{\bf{Step 2}}. Let $a\neq 1$. Define $v(t,x):=u(a^2t,ax)$. Then $v_t+A^{ij}(a^2t)v_{x^ix^j}=0$ in $Q_{\lambda
r/a}(1)\cap\Omega$. As easy to check,
$$
|Q_{r/a}(1)|=a^{-\theta-p-2}|Q_r(a)|, \quad (v_{xx})_{Q_{r/a}(1)}=a^2(u_{xx})_{Q_r(a)}, \quad |Q_{\lambda r/a}(1)\cap \Omega|=a^{-\theta-p-2}|
Q_{\lambda r}(a)\cap \Omega|,
$$
and consequently
$$
\aint_{Q_{r/a}(1)}|v_{xx}(t,x)-(v_{xx})_{Q_{r/a}(1)}|^p(x^1)^{\theta-d+p}dxdt=a^{2p}\aint_{Q_{r}(a)}|u_{xx}(t,x)-(u_{xx})_{Q_{r}(a)}|^p(x^1)^{\theta-d+p}dxdt,
$$
$$
\aint_{Q_{\lambda r/a}(1)\cap
\Omega}|v_{xx}(t,x)|^p(x^1)^{\theta-d+p}dxdt=a^{2p}\aint_{Q_{\lambda r}(a)\cap
\Omega}|u_{xx}(t,x)|^p(x^1)^{\theta-d+p}dxdt.
$$
It follows
\begin{eqnarray*}
&&\aint_{Q_{r}(a)}|u_{xx}(t,x)-(u_{xx})_{Q_{r}(a)}|^p(x^1)^{\theta-d+p}dxdt\\&=& a^{-2p}\aint_{Q_{\lambda r/a}(1)\cap
\Omega}|v_{xx}(t,x)|^p(x^1)^{\theta-d+p}dxdt\\
&\leq& a^{-2p}\cdot \frac{N}{(1+\lambda r/a)^p}\;\aint_{Q_{\lambda r/a}(1)\cap
\Omega}|v_{xx}(t,x)|^p(x^1)^{\theta-d+p}dxdt\\
&=&\frac{N}{(1+\lambda r/a)^p}\;\aint_{Q_{\lambda r}(a)\cap
\Omega}|u_{xx}(t,x)|^p(x^1)^{\theta-d+p}dxdt.
\end{eqnarray*}
The theorem is proved.
\end{proof}

\begin{remark}
Note that Theorem \ref{101004.16.53} is  based on Lemma \ref{101004.14.53}. It follows from Remark \ref{2011.1.15} and  Remark \ref{remark 55555} that if $p\geq 2$ then
Theorem \ref{101004.16.53} holds for any $\theta\in (d-1,d+1)$  (not only for $\theta\in (d-1,d]$).
Obviously we cannot use this result yet since Remark \ref{2011.1.15} is valid only after we prove Theorem \ref{main theorem}.
\end{remark}

\begin{lemma}
                        \label{lemma 06.08.1}
                        Assume  $\theta\in (d-1,d]$ if $p\in (2,\infty)$ and $\theta\in (d-p+1,d]$ if $p\in (1,2]$. Denote $q:=\theta-d+p$ which is in $(1,p]$.

 (i) Let $u\in C^{\infty}_0(\bR\times \bR^d_+)$ and $f:=u_t+A^{ij}(t)u_{x^ix^j}$. Suppose that $A^{ij}(t)$ is  infinitely differentiable and has bounded derivatives.
   Then for any $\varepsilon>0$,
$Q_r(t_0,a,x'_0)\subset\Omega$ and $(t,x)\in Q_r(t_0,a,x'_0)$
\begin{equation}
                      \label{eqn 11.11}
\aint_{Q_r(t_0,a,x'_0)}|u_{xx}-(u_{xx})_{Q_r(t_0,a,x'_0)}|^q\mu(dyds)\leq \varepsilon \bM(|u_{xx}|^q)(t,x)+N\bM(|f|^q)(t,x),
\end{equation}
where $N=N(\varepsilon,\theta,q,d,d_1,\delta,K)$.

(ii) Furthermore, if $u\in C^{\infty}_0(\bR^d_+)$ and $A^{ij}$ is independent of $t$, then for any $\varepsilon>0$,
$B_r(a,x'_0)\subset \bR^d_+$ and $x\in B_r(a,x'_0)$
\begin{equation}
                      \label{eqn 11.1111}
\aint_{B_r(a,x'_0)}|u_{xx}-(u_{xx})_{B_r(a,x'_0)}|^q\nu(dy)\leq \varepsilon \bM(|u_{xx}|^q)(x)+N\bM(|A^{ij}u_{x^ix^j}|^q)(x),
\end{equation}
where $N=N(\varepsilon,\theta,q,d,d_1,\delta,K)$.
\end{lemma}

\begin{proof}
(i) Without loss of generality we may take $t_0=0$ and $x'_0=0$; $Q_r(t_0,a,x'_0)=Q_r(a)$. In fact, for other cases it is enough to consider the function
$v(t,x):=u(t_0+t,x^1,x'_0+x')$ in place of $u(t,x^1,x')$.

{\bf{Step 1}}. We prove that there exists $\kappa=\kappa(\varepsilon)\in (0,1)$ so that (\ref{eqn 11.11}) holds if  $(r/a)\leq \kappa$.

Let $m$ denote the Lebesque measure on $\bR^{d+1}$.  Assume  $\lambda\geq 4$ and $\lambda r\leq a/4$. Then  $(3a/4) \leq x^1\leq (5a/4)$ if $x^1\in B^1_{\lambda r}(a)$, and therefore
$$
(3/5)^{p+\theta-d}\frac{dtdx}{m(Q_{r}(a))}\leq \frac{\mu(dtdx)}{|Q_r(a)|}\leq (5/3)^{p+\theta-d}\frac{dtdx}{m(Q_{r}(a))} \quad \quad \text{on}\quad Q_r(a),
 $$
 $$
(3/5)^{p+\theta-d}\frac{dtdx}{m(Q_{\lambda r}(a))}\leq \frac{\mu(dtdx)}{|Q_{\lambda r}(a)|}\leq (5/3)^{p+\theta-d}\frac{dtdx}{m(Q_{\lambda r}(a))} \quad \quad \text{on}\quad Q_{\lambda r}(a).
 $$
Denote  $c_0:=(5/3)^{p+\theta-d}$. By Theorem \ref{thm 5.5.1},
 \begin{eqnarray*}
 &&\aint_{Q_r(a)}|u_{xx}-(u_{xx})_{Q_r(a)}|^q\mu(dsdy)\\
  &\leq& \int_{Q_r(a)} \int_{Q_r(a)}|u_{xx}(s,y)-u_{xx}(\tau,\xi)|^q \frac{\mu(dsdy)}{|Q_r(a)|}\frac{\mu(d\tau d\xi)}{|Q_r(a)|}\\
 &\leq& c^2_0\int_{Q_r(a)}\int_{Q_r(a)}|u_{xx}(s,y)-u_{xx}(\tau,\xi)|^q\frac{dsdy}{m(Q_r(a))}\frac{d\tau d\xi}{m(Q_r(a))}\\
 &\leq& Nc^2_0\lambda^{d+2}\int_{Q_{\lambda r}(a)}|f|^q \frac{dyds}{m(Q_{\lambda r}(a))} +Nc^2_0\lambda^{-q}\int_{Q_{\lambda r}(a)}|u_{xx}|^q \frac{dyds}{m(Q_{\lambda r}(a))}\\
 &\leq&Nc^3_0\lambda^{d+2}\aint_{Q_{\lambda r}(a)}|f|^q \mu(dyds) +Nc^3_0\lambda^{-q}\aint_{Q_{\lambda r}(a)}|u_{xx}|^q \mu(dyds)\\
 &\leq&N \lambda^{d+2}\bM(|f|^q)(t,x)+N\lambda^{-q}\bM(|u_{xx}|^q)(t,x),
 \end{eqnarray*}
where $N$ depends only on $d,d_1, p,\theta, \delta,K$. Note that the above inequality holds as long as $r\lambda/a \leq 1/4$. Now we fix $\lambda$ so that $N\lambda^{-q} = \varepsilon/2$, i.e. $\lambda=(2N/\varepsilon)^{1/q}$ and define $\kappa=1/{(4\lambda)}=1/4 \cdot (2N/\varepsilon)^{-1/q}$. Then whenever
$r/a\leq \kappa$ we have $(r/a)\lambda\leq 1/4$ and thus (\ref{eqn 11.11})  follows.

{\bf{Step 2}}.
 For given $\varepsilon$, take  $\kappa=\kappa(\varepsilon)$ from Step 1. Assume $r/a\geq  \kappa$. Choose $\lambda$, which will be specified later, so that $r\lambda>4a$;  this  $\lambda$ is different from the one in step 1.  Take a $\zeta\in C^{\infty}_0(\bR^{d+1})$
so that $\zeta(t,x)=1$ for $(t,x)\in Q_{\lambda r/2}(a)\cap \Omega$ and  $\zeta(t,x)=0$ if $(t,x)\not\in (-\lambda^2r^2,\lambda^2r^2)\times (-a,a+\lambda r)\times B'_{\lambda r}$.  Denote
$$
g=f\zeta, \quad h=f(1-\zeta).
$$
Take a large $T$ so that  $u(t,x)=0$ if $t\geq T$.
By Lemma \ref{lemm 12.08.7} we can define  $v$ as the solution of
\begin{equation}
                       \label{eqn again}
v_t+A^{ij}v_{x^ix^j}=h, \quad t\in (S,T), \quad v(T,\cdot)=0
\end{equation}
so that $v\in \frH^n_{p,d}(S,T)$ for any  $n$ and $S>-\infty$.  Also let $\bar{v}\in  \frH^n_{p,d}(S,T+1)$ be the solution of
$$
\bar{v}_t+A^{ij}\bar{v}_{x^ix^j}=h, \quad t\in (S,T+1), \quad \bar{v}(T+1,\cdot)=0.
$$
Then by considering the equation for $\bar{v}$  on $(T,T+1)$, since $h(t)=0$ for $t\geq T$, we conclude $\bar{v}(t)=0$ for $t\in [T,T+1]$. Thus $\bar{v}$ also satisfies (\ref{eqn again}) and $v=\bar{v}$.
It follows from (\ref{eqn 3.31.4}) that $v$ is infinitely differentiable in $x$ (and hence in $t$) in $\Omega$.
 By applying Theorem \ref{101004.16.53} with $\bar{p}=q,\bar{\theta}=d$ and $\lambda/2$ in places of $p,\theta$ and $\lambda$ respectively,
\begin{eqnarray}
               \aint_{Q_r(a)}|v_{xx}(t,x)-(v_{xx})_{Q_r(a)}|^q \bar{\mu}(dyds) &\le& N
\frac{1}{(1+\lambda r/2a)^q}\aint_{Q_{\lambda r/2}(a)\cap \Omega}|v_{xx}(t,x)|^q\bar{\mu}(dyds) \nonumber\\
&\le& N
\frac{1}{(1+\lambda r/a)^q}\aint_{Q_{\lambda r}(a)\cap \Omega}|v_{xx}(t,x)|^q\bar{\mu}(dyds),  \label{eqn 6.08.2}
\end{eqnarray}
where $\bar{\mu}(dsdy):=(y^1)^{\bar{\theta}-d+\bar{p}}dyds=(y^1)^{q}dyds=\mu(dyds)$.
On the other hand, $w:=u-v$ satisfies $w(T,\cdot)=0$ and
$$
w_t+A^{ij}w_{x^ix^j}=g, \quad t\in (0,T).
$$
    By Lemma \ref{lemm 12.08.7},
$$
\int_{Q_r(a)}|w_{yy}|^q (y^1)^{q}dyds\leq \int_{Q_{\lambda r}(a)\cap \Omega} |w_{yy}|^q (y^1)^{q}dyds \leq N\int_{Q_{\lambda r}(a)\cap \Omega}|f|^q(y^1)^{q}\;dyds,
$$
\begin{eqnarray}
\aint_{Q_r(a)}|w_{yy}|^q\mu(dyds)&\leq& N\frac{\lambda^{d+1}(1+\lambda r/a)^{p+\theta-d+1}}{(1+r/a)^{p+\theta-d+1}-(1-r/a)^{p+\theta-d+1}}\aint_{Q_{\lambda r}(a)\cap \Omega}|f|^q\mu(dyds) \nonumber\\
&\leq &N(\kappa)\lambda^{d+1}(1+\lambda r/a)^{p+\theta-d+1}\aint_{Q_{\lambda r}(a)\cap \Omega}|f|^q\mu(dyds) \label{eqn 6.08.1},
\end{eqnarray}
where for the second inequality we use $(1+r/a)^{p+\theta-d+1}-(1-r/a)^{p+\theta-d+1}\geq (1+\kappa)^{p+\theta-d+1}-1$.
Observing that $u=v+w$,
\begin{eqnarray*}
I:&=&\aint_{Q_r(a)}|u_{yy}(t,x)-(u_{yy})_{Q_r(a)}|^q \mu(dyds)\\
&\leq& N(q)\aint_{Q_r(a)}|w_{yy}(t,x)-(w_{yy})_{Q_r(a)}|^q \mu(dyds)+N(q)\aint_{Q_r(a)}|v_{yy}(t,x)-(v_{yy})_{Q_r(a)}|^q \mu(dyds)\\
&\leq& N(q)\aint_{Q_r(a)}|w_{yy}(t,x)|^q \mu(dyds)+N(q)\aint_{Q_r(a)}|v_{yy}(t,x)-(v_{yy})_{Q_r(a)}|^q \mu(dyds)
\end{eqnarray*}
and thus  by (\ref{eqn 6.08.2}) and (\ref{eqn 6.08.1}),
\begin{eqnarray*}
I  &\leq&N\lambda^{d+1}(1+\lambda r/a)^{p+\theta-d+1}\aint_{Q_{\lambda r}(a)\cap \Omega}|f|^q\mu(dyds)\\&&+ N
\frac{1}{(1+\lambda r/a)^q}\aint_{Q_{\lambda r}(a)\cap \Omega}|v_{yy}(t,x)|^q\mu(dyds)\\
&\leq& N\lambda^{d+1}(1+\lambda r/a)^{p+\theta-d+1}\aint_{(0,\lambda^2r^2)\times(0,a+\lambda r)}|f|^q\mu(dyds)\\
&&+ N\frac{1}{(1+\lambda r/a)^q}\aint_{Q_{\lambda r}(a)\cap \Omega}\left(|u_{yy}(t,x)|^q+|w_{yy}(t,x)|^q\right)\mu(dyds)\\
&\leq& N\lambda^{d+1}(1+\lambda r/a)^{p+\theta-d+1}\aint_{Q_{\lambda r}(a)\cap \Omega}|f|^q\mu(dyds)\\
&&+ N\frac{1}{(1+\lambda r/a)^q}\aint_{Q_{\lambda r}(a)\cap \Omega}|u_{yy}(t,x)|^q\mu(dyds).
\end{eqnarray*}
Now  to prove the first assertion it is enough to   choose $\lambda$ so large that $N\frac{1}{(1+\lambda r/a)^q}\leq \varepsilon$. Also note that since $r/a\geq \kappa$, we have
$$
N \lambda^{d+1}(1+\lambda r/a)^{p+\theta-d+1}\leq N(\lambda,\kappa).
$$

(ii) The second assertion is proved similarly based on Corollary \ref{cor 5.17.1} and (\ref{eqn 5.13.2011}) in place Theorem \ref{thm 5.5.1} and (\ref{eqn 2.21.2011}).  The lemma is proved.

\end{proof}


\mysection{Proof of Theorem \ref{main theorem} and Theorem \ref{main theorem-elliptic}}
                                                   \label{section proof}

Firstly, we give an $L_p$-theory   for  the following backward system defined on $\bR\times \bR^d_+$.

\begin{theorem}
                 \label{thm all time}
  Let $p\in (1,\infty)$. Assume $\theta\in (d-1,d+1)$ if $p\in (2,\infty)$, and $\theta\in (d+1-p, d+p-1)$ if $p\in (1,2]$. Then for any $f\in \bL_{p,\theta}(-\infty,\infty)$ the system
  $$
  u_t +A^{ij}(t)u_{x^ix^j}=f
  $$
  has a unique solution $u$ in $M\bH^2_{p,\theta}(-\infty,\infty)$ and for this solution we have
  \begin{equation}
                        \label{main apriori}
  \|Mu_t\|_{\bL_{p,\theta}(-\infty,\infty)}+\|M^{-1}u\|_{\bH^2_{p,\theta}(-\infty,\infty)}\leq N\|Mf\|_{\bL_{p,\theta}(-\infty,\infty)}.
  \end{equation}
 \end{theorem}
\begin{proof}
If $A^{ij}u_{x^ix^j}=\Delta u=(\Delta u^1,\ldots,\Delta u^{d_1})$, then the theory of single equations is applied and
the theorem is true for any $\theta\in (d-1,d-1+p)$; see Theorem 5.6 in \cite{kr99}.
Actually the mentioned theorem is proved for parabolic equations defined on $(0,T)\times \bR^d_+$, but one can easily check that
the proofs in \cite{kr99} work for equations defined on $\bR \times \bR^d_+$.

 For
$\lambda\in [0,1]$ and $d_1\times d_1$ identity matrix $I$ we define
$$
{A}^{ij}_{\lambda}=({a}^{ij}_{kr,\lambda})
:=(1-\lambda)A^{ij}+\delta^{ij}\lambda
\delta I.
$$
Then for each $\lambda\in [0,1]$ the coefficient matrices $\{A^{ij}_{\lambda}:i,j=1,\ldots,d\}$ satisfy Assumption \ref{main assumption2} with the same $\delta,K$.
Thus due to the method of continuity, we only need to prove that a priori estimate (\ref{main apriori}) holds given that a solution $u$ already exists. Furthermore, since $C^{\infty}_0(\bR\times \bR^d_+)$ is
 dense in $M\bH^2_{p,\theta}(-\infty,\infty)$, we may assume that $u\in C^{\infty}_0(\bR\times \bR^d_+)$.  By Remark \ref{remark 3.21.6},  we only need to prove
 the following:
 \begin{equation}
                     \label{very important}
 \int^{\infty}_{-\infty}\int_{\bR^d_+}|u_{xx}(t,x)|^p \mu(dtdx)\leq N  \int^{\infty}_{-\infty}\int_{\bR^d_+}|f(t,x)|^p \mu(dtdx).
 \end{equation}
To prove this we certainly may assume that $A^{ij}$ are infinitely differentiable and have bounded derivatives (remember that the constant $N$ in (\ref{eqn 11.11}) do not depend on the regularity of $A^{ij}$).

{\bf{Case 1}}. Assume that either (i) $p\in (2,\infty)$ and $\theta\in (d-1,d]$  or (ii) $p\in (1,2]$ and $\theta\in (d-p+1,d]$.

Define $q:=\theta-d+p$. Recall  that the range of $q\in (1,p]$.  By Lemma \ref{lemma 06.08.1}, if $u\in C^{\infty}_0(\bR\times \bR^d_+)$, then for any $\varepsilon>0$
$$
(u_{xx})^{\sharp}(t,x)\leq \varepsilon \bM^{1/q}(|u_{xx}|^q)(t,x)+N(\varepsilon)\bM^{1/q}(|u_t+A^{ij}u_{x^ix^j}|^q)(t,x).
$$
By Theorem \ref{FS}
$($Fefferman-Stein$)$ and Theorem \ref{HL} $($Hardy-Littlewood$)$,
\begin{eqnarray*}
\|Mu_{xx}\|_{\bL_{p,\theta}(-\infty,\infty)}&=&\|u_{xx}\|_{L_p(\Omega, \mu)}\\&\leq& N \|(u_{xx})^{\sharp}\|_{L_p(\Omega, \mu)}\\
&\leq&N \varepsilon\|\bM^{1/q}(|u_{xx}|^q)\|_{L_p(\Omega, \mu)}+N\cdot N(\varepsilon)\|\bM^{1/q}(|u_t+A^{ij}u_{x^ix^j}|^q)\|_{L_p(\Omega, \mu)}\\
&=&N \varepsilon \|\bM(|u_{xx}|^q)\|^{1/q}_{L_{p/q}(\Omega, \mu)}+N\cdot N(\varepsilon)\|\bM(|u_t+A^{ij}u_{x^ix^j}|^q)\|^{1/q}_{L_{p/q}(\Omega, \mu)}\\
&\leq&N \varepsilon \||u_{xx}|^q\|^{1/q}_{L_{p/q}(\Omega, \mu)}+N\cdot N(\varepsilon)\||u_t+A^{ij}u_{x^ix^j}|^q\|^{1/q}_{L_{p/q}(\Omega, \mu)}\\
&=&N \varepsilon \|u_{xx}\|_{L_{p}(\Omega, \mu)}+N\cdot N(\varepsilon)\|u_t+A^{ij}u_{x^ix^j}\|_{L_{p}(\Omega, \mu)}.
\end{eqnarray*}
This obviously yields (\ref{very important}).

{\bf{Case 2}}.  Assume that either (i) $p \in (2,\infty)$ and $\theta\in [d,d+1)$  or (ii) $p\in (1,2]$ and $\theta\in [d,d+p-1)$.
 By Remark \ref{remark 3.21.6}  we only need to prove
 the following:
 \begin{equation}
                     \label{very important22}
 \int^{\infty}_{-\infty}\int_{\bR^d_+}|M^{-1}u(t,x)|^p (x^1)^{\theta-d}dxdt\leq N  \int^{\infty}_{-\infty}\int_{\bR^d_+}|Mf(t,x)|^p (x^1)^{\theta-d}dxdt.
 \end{equation}
To prove this, we use a duality (Lemma \ref{duality}).  Denote $p'=p/(p-1)$ and choose  $\bar{\theta}$ so that $\theta/p+\bar{\theta}/p'=d$.
 Then  $\bar{\theta}\in (d-1,d]$ if $p'\in (2,\infty)$ and  $\bar{\theta}\in (d-p'+1,d]$ if  $p'\in (1,2]$.

 Changing the variable $t\to -t$ shows that the result of case 1 is applicable to the operator $u_t-A^{ij}u_{x^ix^j}$ in place of $u_t+A^{ij}u_{x^ix^j}$. Therefore for any $v\in M\bH^2_{p',\bar{\theta}}(-\infty,\infty)$,
by integration by parts,
\begin{eqnarray*}
\int_{\bR^{d+1}_+} M^{-1}uM(v_t-A^{ij}v_{x^ix^j}) dxdt&=&\int_{\bR^{d+1}_+} u(v_t-A^{ij}v_{x^ix^j}) dxdt\\
&=&\int_{\bR^{d+1}_+} M(-u_t-A^{ij}u_{x^ix^j})M^{-1}v dxdt\\
&\leq& \|M(u_t+A^{ij}u_{x^ix^j})\|_{\bL_{p,\theta}(-\infty,\infty)}\|M^{-1}v\|_{\bL_{p',\bar{\theta}}(-\infty,\infty)}\\
&\leq&N\|M(u_t+A^{ij}u_{x^ix^j})\|_{\bL_{p,\theta}(-\infty,\infty)}\|M(v_t-A^{ij}v_{x^ix^j})\|_{\bL_{p',\bar{\theta}}(-\infty,\infty)}.
\end{eqnarray*}
Since, by Case 1, $\{v_t-A^{ij}v_{x^ix^j}: v\in M\bH^2_{p',\bar{\theta}}(-\infty,\infty)\}$ is dense in $M^{-1}\bL_{p',\bar{\theta}}(-\infty,\infty)$, it follows that
$$
\|M^{-1}u\|_{\bL_{p,\theta}(-\infty,\infty)}\leq N \|M(u_t+A^{ij}u_{x^ix^j})\|_{\bL_{p,\theta}(-\infty,\infty)}.
$$
The theorem is proved.
\end{proof}

{\bf{Proof of Theorem \ref{main theorem}}}\quad
As usual, we assume $u_0=0$. For details see the proof of Theorem 5.1 in \cite{Kr99}.

{\bf{Case 1}}. Let $T=\infty$.  As before we only prove the a priori estimate. Suppose $u\in \frH^{\gamma+2}_{p,\theta}(\infty)$ satisfies
\begin{equation}
                      \label{eqn infi}
u_t=A^{ij}u_{x^ix^j}+f, \quad t \in (0,T)\, ; \quad u(0,\cdot)=0.
\end{equation}
Define $v(t,x)=u(t,x)I_{t>0}$ and $\bar{f}=fI_{t>0}$, then  $v\in M^{-1}\bH^{2}_{p,\theta}(-\infty,\infty)$ and $v$ satisfies (see Definition \ref{definition 5.8.1})
$$
v_t=A^{ij}u_{x^ix^j}+\bar{f}, \quad (t,x)\in \bR^{d+1}_+.
$$
By Theorem \ref{thm all time},
$$
\|Mu_{xx}\|_{\bL_{p,\theta}(\infty)}\leq N\|Mf\|_{\bL_{p,\theta}(\infty)}.
$$
By  Remark \ref{remark 3.21.6}, this certainly proves (\ref{a priori}).

{\bf{Case 2}}. \quad Let $T<\infty$. The existence of the solution in $\frH^{\gamma+2}_{p,\theta}(T)$ is obvious. Now suppose that $u\in \frH^{\gamma+2}_{p,\theta}(T)$ is a solution of (\ref{eqn infi}).
By  the result of Case 1, the system
\begin{equation}
                       \label{eqn 3.07.1}
v_t=\Delta v +(A^{ij}u_{x^ix^j}+f-\Delta u)I_{t\leq T},  \quad t>0\,; \quad v(0,\cdot)=0
\end{equation}
has a unique solution $v\in \frH^{\gamma+2}_{p,\theta}(0,\infty)$. Then $v-u$ satisfies
$$
(v-u)_t=\Delta(v-u), \quad t\in (0,T)\,; \quad (v-u)(0,\cdot)=0.
$$
If follows from the theory of single equations (see, for instance, Theorem 5.6 in \cite{kr99}), $u=v$ for $t\in [0,T]$.  For $t\geq 0$, define
$$
A^{ij}_T=(a^{ij}_{T,kr}), \quad a^{ij}_{T,kr}=a^{ij}_{kr}I_{t\leq T}+ \delta^{ij} \delta^{kr} I_{t>T}.
$$
Then (\ref{eqn 3.07.1}) and the fact $u=v$ for $t\in [0,T]$ show that $v$ satisfies (replace $u$ by $v$ for $t\leq T$ in (\ref{eqn 3.07.1}))
\begin{equation}
                           \label{eqn 4.25.1}
v_t=A^{ij}_Tv_{x^ix^j}+fI_{t<T}, \quad t>0\,;\,\, v(0,\cdot)=0.
\end{equation}
By Case 1, $v\in \frH^{\gamma+2}_{p,\theta}(\infty)$ is the unique solution of (\ref{eqn 4.25.1}), and $u=v$ on $[0,T]$ whenever $u$ is a solution of  (\ref{eqn infi}) on $[0,T]$.
This obviously yields the uniqueness. The theorem is proved. \hspace{3cm}$\Box$

\vspace{3mm}

{\bf{Proof of Theorem  \ref{main theorem-elliptic}}}  \,\,\,
 The proof  is very similar to that of the proof of Theorem \ref{main theorem} and is based on (\ref{eqn 11.1111}). We leave the details to the readers as an exercise.







\begin{thebibliography}{mm}
{\small


\bibitem{B} M. Bramanti and M.C. Ceruti, {\em $W^{1,2}_p$ solvability for the Cauchy-Dirichlet problem for parabolic equations with VMO doeffieients\/},
Comm. Partial Differential Equations, {\bf{18}} (1993), no. 9-10, 1735-1763.

\bibitem{By} Sun-Sig Bun, {\em Parabolic equations with BMO coefficients in Lipschitz domains\/}, J. Differential Equations, {\bf{209}} (2005), no. 2, 229-265.

\bibitem{C} F. Chiarenza, M. Frasca and P. Longo, {\em $W^{2,p}$-solvability of the Dirichlet problem for nondivergence elliptic equations
 with VMO coefficients\/}, Trans. Amer. Math. Soc., {\bf{336}} (1993), no. 2, 841-853.

\bibitem{Do} D. Kim, {\em Elliptic equations with nonzero boundary condtions in weighted Sobolev spaces\/}, J. Math. Anal. Appl. {\bf{337}} (2008), 1465-1479.

\bibitem{G} P. Grisvard, {\em Elliptic Problems in nonsmooth domains\/}, Monographs and Studies in Mathematics {\bf{24}} (1985), Pittman, Boston-London-Melbourn.

\bibitem{H} R. Haller-Dintelmann, H. Heck and M. Hieber, {\em $L^p-L^q$-estimates for parabolic systems in non-divergence  form with VMO coefficients\/}, J. London Math. Soc. (2) {\bf{74}} (2006), no. 3, 717-736.

\bibitem{Doy} Doyoon Kim and N.V. Krylov, {\em Parabolic equations with measurable coefficients\/}, Potential Anal. {\bf{26}} (2007), no. 4, 345-361.




\bibitem{Kr01}  N.V. Krylov, {\em Some properties of traces for stochastic
and determistic parabolic weighted Sobolev spaces}, Journal of
Functional Analysis {\bf{183}} (2001), 1-41.



\bibitem{Kr99}  N.V. Krylov, {\em An analytic approach to SPDEs\/},
 pp. 185-242 in
Stochastic Partial Differential Equations: Six Perspectives,
Mathematical Surveys and Monographs,  {\bf{64}} (1999), AMS,
Providence, RI.

\bibitem{kr99} N.V. Krylov, {\em Weighted Sobolev spaces and Laplace equations
and the heat equations in a half space\/}, Comm. in PDEs, {\bf{23}}
(1999), no. 9-10, 1611-1653.


\bibitem{kr99-1} N.V. Krylov, {\em Some properties of weighted Sobolev spaces in $\bR^d_+$\/}, Ann. Scuola Norm. Sup. Pisa Cl. Sci. {\bf{28}} (1999), no. 4, 675-693.
\bibitem{kr94} N.V. Krylov, {\em A $W^n_2$-theory of the Dirichlet
problem for SPDEs in general smooth domains\/}, Probab.Theory
Relat.Fields {\bf{98}}(1994), 389-421.

\bibitem{kr08} N.V. Krylov, {\em Lectures on Elliptic and Parabolic Equations in Sobolev Spaces\/},
American Mathematical Society, Prividence, RI, 2008.


\bibitem{KL2} N.V. Krylov and S.V. Lototsky, {\em A Sobolev space
theory of SPDEs with constant coefficients in a half space\/}, SIAM
J. on Math. Anal., {\bf{31}} (1999), no. 1, 19-33.

\bibitem{Lee} Kijung Lee, {\em On a Deterministic Linear Partial
Differential System\/}, Journal of Mathematical Analysis and
Applications, {\bf{353}}, (2009), no. 1, 24-42.


\bibitem{L} J.-L. Lions and E. Magenes, {\em Probl\`{e}mes aux limites non homog\'enes et applications $1$\/}, Dunod, Paris, 1968.




\bibitem{Lo2} S.V. Lototsky,
{\em Sobolev spaces with weights in domains and boundary value
problems for degenerate elliptic equations\/}, Methods
 and Applications of
Analysis, {\bf{1}} (2000), no. 1, 195-204.



\bibitem{S} V.A. Solonnikov, {\em Solvability of the classical initial-boundary-value problems for the heat-conduction equations in a dihedral angle\/}, Zapiski Nauchnykh Seminarov LOMI {\bf{138}} (1984), 146-180 in Russian;
English translation in Journal of Soviet Math. {\bf{32}}  (1986), no. 5, 526-546.

\bibitem{T} H. Triebel, {\em Theory of function spaces\/},
Birkh\"auser Verlag, Basel-Boston-Stuttgart, 1983.


}
\end{thebibliography}
\end{document}